\documentclass[11pt]{article}

\usepackage{amsmath, amssymb,latexsym,graphicx,mfpic}
\usepackage{verbatim}
\usepackage[all]{xy}
\usepackage{youngtab}

\def\F{\mathbb{F}}

\def\N{\mathbb{N}}
\def\Q{\mathbb{Q}}
\def\Z{\mathbb{Z}}

\def\id{{\rm id}}

\def\Aut{{\rm Aut}}

\def\Hom{{\rm Hom}}
\def\sgn{{\rm sgn}}

\def\Ind{{\rm Ind}}
\def\Gal{{\rm Gal}}

\def\Disc{{\rm Disc}}
\def\GL{{\rm GL}}
\def\Spec{{\rm Spec}\,\,}
\def\triv{{\rm triv}}
\def\std{{\rm std}}
\def\sgn{{\rm sgn}}

\def\Tr{{\rm Tr}}
\def\bij{{\rm Bij}}
\def\Hom{{\rm Hom}}

\def\O{{\mathcal O}}

\def\binom#1#2{{#1 \choose #2}}

\newtheorem{theorem}{Theorem}
\newtheorem{lemma}[theorem]{Lemma}

\newtheorem{proposition}[theorem]{Proposition}

\newtheorem{definition}[theorem]{Definition}
\newtheorem{remark}[theorem]{Remark}
\newtheorem{question}{Question}

\newenvironment{theorem1'}{\noindent {\bf Theorem 1${}^{\prime}$}}

\newenvironment{theorem3'}{\noindent {\bf Theorem 3${}^{\prime}$}}

\newenvironment{theorem4'}{\noindent {\bf Theorem 4${}^{\prime}$}}

\newenvironment{theorem5'}{\noindent {\bf Theorem 5${}^{\prime}$}}

\newcommand{\tr}{\triangleright}
\newcommand{\OO}{\mathcal{O}}

\newenvironment{proof}{\noindent {\bf Proof:}}{$\Box$ \vspace{2 ex}}
\newenvironment{proofl}{\noindent {\bf Proof of Lemma \ref{l:incexc}:}}{$\Box$ \vspace{2 ex}}
\newenvironment{prooft}{\noindent {\bf Proof of Theorem \ref{cor:reg}:}}{$\Box$ \vspace{2 ex}}
\newenvironment{prooft2}{\noindent {\bf Proof of Theorem \ref{newth}:}}{$\Box$ \vspace{2 ex}}

\title{{On a notion of ``Galois closure'' for extensions of rings}}
  \author{Manjul Bhargava and Matthew Satriano}
\begin{document}
\maketitle 

\section{Introduction}
%In our first two articles [] and [], we developed laws of composition
%on various spaces of forms.

Let $A$ be any {ring of rank $n$ over a base ring $B$}, i.e., a
$B$-algebra that is free of rank $n$ as a $B$-module.  In this
article, we investigate a natural definition for the ``Galois
closure'' $G(A/B)$ of the ring $A$ as an extension of
$B$.\footnote{All rings are assumed to be commutative with unity.}

The definition is as follows.  For an element $a\in A$, let
\begin{equation}\label{pa}
P_a(x)=x^n-s_1(a)x^{n-1}+s_2(a) x^{n-2}+\cdots+(-1)^n s_n(a)
\end{equation}
be the characteristic polynomial of $a$, i.e., the characteristic polynomial
of the $B$-module transformation $\times a:A\to A$ given by
multiplication by $a$.  Furthermore, for an element $a\in A$, let
$a^{(1)}$, $a^{(2)}$, \ldots, $a^{(n)}$ denote the elements 
$a\otimes 1\otimes 1\otimes\cdots\otimes1$, 
$1\otimes a\otimes 1\otimes\cdots\otimes1$, $\ldots$, 
$1\otimes 1\otimes 1\otimes\cdots\otimes a$ in $A^{\otimes n}$
respectively. Let $I(A,B)$ denote the ideal in $A^{\otimes n}$ generated 
by all expressions of the form
\begin{equation}\label{fundrels}
s_j(a)\;-\sum_{1\leq i_1< i_2< \ldots< i_j\leq n}
a^{(i_1)}a^{(i_2)}\cdots a^{(i_j)}\end{equation}
where $a\in A$ and $j\in\{1,2,\ldots,n\}$.  
Note that the symmetric group $S_n$ naturally acts on $A^{\otimes n}$
by permuting the tensor factors, and the ideal $I(A,B)\subset A^{\otimes n}$
is preserved under this $S_n$-action.  We are interested in imposing on $A^{\otimes n}$ the relations in $I(A,B)$ defined by (\ref{fundrels}) 
because they are precisely the relations that the conjugates $a^{(i)}$ of a generic element $a$ in a separable field extension of degree~$n$ would satisfy in a normal closure.  Alternatively, they are the general relations that the eigenvalues $a^{(i)}$ of a linear transformation of a vector space of dimension $n$ would satisfy.

We define
\begin{equation}\label{gcdef}
G(A/B)=A^{\otimes n}/I(A,B),
\end{equation}
and we call $G(A/B)$ the {\it $S_n$-closure} of $A$ over $B$.  Since $I(A,B)$ is $S_n$-invariant,
we see that the action of $S_n$ on $A^{\otimes n}$ also descends to an
$S_n$-action on $G(A/B)$.  One easily checks (or see Theorem~\ref{fieldcase}
below) that if $A/B$ is a degree $n$ extension of fields having
associated Galois group $S_n$, then $G(A/B)$ is indeed simply the
Galois closure of $A$ as a field extension of $B$.  Thus our
definition of $S_n$-closure in a sense naturally extends the usual
notion of Galois closure to rank $n$ ring extensions.

In fact, our definition 
above also naturally extends to $B$-algebras $A$ that are
{\it locally free of rank}~$n$. A $B$-module $M$ is said to be 
{locally free of rank~$n$} if there exist $b_1,\dots,b_m\in B$ such 
that $\sum Bb_i=B$ and $B_{b_i}\otimes_B M$ is free of 
rank $n$ over the localization $B_{b_i}$.\footnote{The
%We  note (see, e.g., \cite[Thm.~4.6]{Lenstra}) 
%that the
%We note  that the
condition that {\it $M$ is locally free of rank $n$ as a $B$-module} 
is also equivalent to either of the following two natural conditions:
(a) $M$ is finitely generated and projective of constant rank $n$
as a $B$-module; 
(b) $M$ is finitely presented and $M_{\mathfrak{m}}$ is free of rank $n$ as a
$B_{\mathfrak{m}}$-module for all maximal ideals $\mathfrak{m}$ of $B$.
\,(See, e.g., \cite[Thm.~4.6]{Lenstra}.)}
%we can define the characteristic polynomial $P_a$ of an element $a\in A$ as follows.  First, 
For such $M$, we have a natural isomorphism 
\begin{equation}\label{iso} 
M\otimes_{B}\textrm{Hom}_B(M,B)\to \textrm{End}_B(M), 
\end{equation} 
where for $B$-modules $N,N'$ we use $\textrm{Hom}_B(N,N')$ to 
denote the set of $B$-module homomorphisms from $N$ to $N'$, and we use
$\textrm{End}_B(N)$ to denote $\textrm{Hom}_B(N,N)$.  Indeed, (\ref{iso}) 
gives an isomorphism locally on $B_{b_i}$ (since $B_{b_i}\otimes_B M$ is free over 
$B_{b_i}$), 
and hence it is an isomorphism globally.  Next, if $f$ is any $B$-module 
endomorphism of $M$, then the trace of $f$ is defined to be the image of 
$f$ under the canonical map 
\[ \Tr:\textrm{End}_B(M)\cong 
M\otimes_{B}\textrm{Hom}_B(M,B)\to B. \] 
%the first isomorphism holds by \cite[??]{}.  
Finally, if $A$ is a $B$-algebra which is locally free of rank $n$, then 
given an element $a\in A$, we obtain a 
$B$-module endomorphism of $A$ given by $\times a:A\to A$. We let $s_j(a)$ 
be the trace of the induced $B$-module endomorphism of $\bigwedge^j A$.  
%, which is locally projective by \cite[??]{}.  
Note that for such $A$, it makes sense to speak of the characteristic polynomial 
$P_a$ of an element $a\in A$, and that the {\it Cayley-Hamilton
Theorem} carries over to this setting as $P_a(a)$ is locally zero, hence 
globally zero.  We can then define $I(A,B)$ and $G(A/B)$ as in 
(\ref{fundrels}) and (\ref{gcdef}).

The notion of $S_n$-closure has a number of interesting properties,
which we consider in this article.  First, we note that the 
%first property that should be
%mentioned is that the 
$S_n$-closure construction is clearly functorial in $A$ for $B$-algebra morphisms.
%, i.e., if $A$ and $A'$ are 
%rings of rank $n$ over a base ring $B$, then any morphism $A\to A'$ of $B$-algebras 
%induces a morphism $G(A/B)\to G(A'/B)$ of $B$-algebras in the natural way.  Furthermore, 
The first nontrivial property that should be mentioned is that the $S_n$-closure construction commutes with base change:

\begin{theorem}
\label{thm:functorial}
If $A$ is a ring of rank $n$ over $B$, and $C$ is a $B$-algebra, then there is a 
natural isomorphism
\[
G(A/B)\otimes_B C\simeq G((A\otimes_B C)/C)
\]
of $C$-algebras.
%, sending the image of $a^{(1)}\otimes 1$ to the image of $a
\end{theorem}

Next, in the case of an extension of fields, we have

\begin{theorem}\label{fieldcase}
Let $B$ be a field, and suppose $A$ is a separable 
field extension of $B$ of degree $n$.  Let $\widetilde{A}$
be a Galois closure of $A$ over $B$, and let 
$r=\frac{n!}{{\rm deg}({\widetilde{A}/B})}$.  Then 
\[G(A/B)\cong\widetilde{A}^{r}
%\textstyle{\frac{k!}{{\rm deg}(\widetilde{A}/B)}}}.
\]
as $B$-algebras.  
\end{theorem}
%\pagebreak
In particular, if ${\rm deg}(\widetilde{A}/B)=n!$ $($i.e.,
Gal($\widetilde{A}/B)=S_n)$, then $G(A/B)\cong\widetilde{A}$ as 
$B$-algebras.

We next consider the case where $B$ is {\it monogenic} 
over $A$, i.e., $A$ is
generated by one element as a $B$-algebra.  Then we have

\begin{theorem}\label{monocase}
  Suppose $A$ is a ring of rank $n$ over $B$ such that $A=B[\alpha]$
  for some $\alpha\in A$.  
%that is generated by one element $\alpha\in A$ as a $B$-algebra.  
Then 
%the $S_n$-closure
  $G(A/B)$ is a ring of rank $n!$ over $B$.  More generally, if $A$ is
  locally free of rank $n$ over $B$ and is locally generated by one
  element, then $G(A/B)$ is locally free of rank $n!$ over $B$.
\end{theorem}

Now, if $B$ is any ring, then we may examine the ring $A=B^n$
having rank $n$ over $B$.  More generally, we may consider those
locally free rings $A$ of rank $n$ that are {\it \'etale} over $B$,
i.e., those $A$ for which the determinant of the bilinear form $\langle
a,a'\rangle=\Tr(aa')$---called the {\it discriminant} $\Disc(A/B)$ of 
$A$
over $B$---is a unit in $B$ (equivalently, 
%a locally free ring $A$ of rank $n$ over $B$ is \'etale over 
%$B$ if
those $A$ for which the map $\Phi:A\to\Hom_B(A,B)$ given by $a\mapsto
(a'\mapsto\Tr(aa'))$ is a $B$-module isomorphism). 
%(equivalently,
%A ring $A$ over $B$ is called {\it locally free of rank $n$} if there exist $b_1,\dots,b_m\in B$ such that $\sum b_i=1$ 
%and $A\otimes_B B_{b_i}$ is free over the localization $B_{b_i}$.  If $A$ is only locally free of rank $n$ over $B$, then 
%$A$ is {\it \'etale over} $B$ if each $A\otimes_B B_{b_i}$ is \'etale over $B_{b_i}$, where the $b_i$ are as above.  
We prove:

% (the simplest 
%instance being $A=B^n$).  

\begin{theorem}\label{etalecase}
  For any ring $B$, we have $G(B^n/B)\cong B^{n!}$.  If
  $A$ is \'etale and locally free of rank $n$ over $B$, then
  $G(A/B)$ is \'etale and locally free of rank $n!$ over $B$.
\end{theorem}
In fact, if $B$ has no nontrivial idempotents, we may explicitly describe the Galois set 
associated to $G(A/B)$ in terms of that associated to $A$ (see Section \ref{etalesec}).

Thus for either \'etale or locally monogenic 
ring extensions of rank $n$, the $S_n$-closure
construction always yields locally free ring extensions 
of rank $n!$.  For general rings that are locally free of small
rank over a base $B$---even those that might not be \'etale or
(locally)
monogenic---the $S_n$-closure still always yields locally
free rings of rank $n!$ over $B$:

\begin{theorem}\label{cubicase}
Suppose $A$ is locally free of rank $n\leq 3$ over $B$.  Then $G(A/B)$
is locally free of rank~$n!$ over~$B$.  
%More generally, if $A$ is locally free of rank $n\leq 3$ over $B$,
%then $G(A/B)$ is locally 

%free of rank $n!$ over $B$.
\end{theorem}
For example, if one takes an order $A$ in a noncyclic cubic field $K$,
then its $S_3$-closure yields a canonically associated order $\tilde A=G(A/\Z)$
in the sextic field $\widetilde{K}$.  We will prove in Section~7 that this 
sextic order  
satisfies $\Disc(\tilde A/\Z)=\Disc(A/\Z)^3$.
%This canonical sextic order $R$ has a number
%of interesting properties, which we study in Section ?.
%For example, we show that we always have
%the identity $\Disc(R)=\Disc(A)^3$.
%; this order $R$ turns out to have discriminant 
%$\Disc(A)^3$.

%Extending Theorem~\ref{etalecase}, w
We may ask how the notion of $S_n$-closure behaves 
under general products.  We prove:

\begin{theorem}\label{directsum}
If $A_1,\ldots,A_k$ are locally free rings of rank $n_1,\ldots,n_k$, respectively, over $B$, 
then 
\begin{equation}
G(A_1\times\cdots\times A_k/B)\; \cong \;\bigl[(G(A_1/B)\otimes \cdots\otimes G(A_k/B)\bigr]^{\textstyle{n\choose{n_1,\ldots,n_k}}}.
\end{equation}
\end{theorem}
Theorem 6 implies that if $A_1,\ldots,A_k$ are locally free rings of rank $n_1,\ldots,n_k$ over $B$ such that each $A_j$ has $S_{n_j}$-closure over $B$ that is locally free of the expected 
%factorial 
rank $n_j!$, then the product $A=A_1\times\cdots\times A_k$ (which is locally free of rank $n=n_1+\cdots+n_k$ over $B$) also has $S_n$-closure that is locally free of the expected rank 
%$n_1!\,n_2!\cdots n_k!\cdot {n\choose{n_1,\ldots,n_k}}=
$n!$ over $B$.

%$G(A_j/B)$ is locally free of rank $n_j!$ over $B$ for all $j\in\{1,\ldots,k\}$, then $G(A_1\times\cdots\times A_k/B)$ is locally free of rank $n!$ over $B$, where $n=n_1+\ldots+n_k$ is the locally free rank of $A=A_1\times\cdots\times A_k$ over $B$.  Hence 

One might imagine that for more complicated ring extensions, however, the
analogues of the rank assertions in Theorems~\ref{monocase}--\ref{cubicase} might not
hold.  Indeed, one finds in rank 4 that there exist algebras over fields
for which the $S_4$-closure need not have rank $4!=24$.  For instance,
we will show in Section \ref{degexample} that 
the $S_4$-closure of the ring $K[x,y,z]/(x,y,z)^2$ has
dimension 32 over $K$ for any field~$K$.  

This has consequences over $\Z$ 
as well.  For example, suppose $K$ is
a quartic field and $A$ is the ring of integers in $K$.  Consider the
suborder $A'=\Z+pA$ for some prime $p$.  Since $A'/pA'\cong\F_p[x,y,z]/(x,y,z)^2$, we see
already that the minimal number of generators for $G(A'/\Z)$ 
as an abelian group is at least 32 by Theorem~1.  Since
$A'\otimes\Q=K$, we see that the torsion-free rank of $A'$ is $4!=24$,
but one finds that there are also eight dimensions of $p$-torsion!
Although this may seem unsightly at first, for a number of reasons
this additional information contained in the $p$-torsion is important to
retain in studying the ``Galois closure'' of the order $A'$ (the most
prominent reason being perhaps the property of commuting with base change.)
%that it makes the Galois closure 
%construction commute with base change).  
We study this example more carefully in Section~\ref{orderexample}.
The example will illustrate that there is no natural further quotient of $G(A'/\Z)$ that has 
24 generators as a $\Z$-module and also respects base change (see Theorem~\ref{newth}). This gives
further evidence that allowing the rank to be higher than $n!$ when
constructing $S_n$-closures can be important when considering somewhat
more ``degenerate'' ring extensions.
%%%%%%%%%%%%%

\begin{remark}{\em It is possible to obtain a natural Galois
    closure-type object of rank $n!$ for any order $A$ in a degree $n$
    number field $K$, by constructing $G(A/\Z)$ as defined above, and
    then quotienting by all torsion.  This quotient was used for
    convenience in, e.g., \cite{Bhargava3} and \cite{Bhargava4}.
    Although quite convenient in many contexts, such a quotienting
    procedure will NOT commute with base change! }
\end{remark}

It is an interesting question as to what the possible dimensions are
for the $S_n$-closure of a dimension $n$ algebra over a field $K$.
In Section~11, we 
show that the largest possible dimensions occur for the ``maximally
degenerate'' rank $n$ algebra over $K$, namely
$R_n=K[x_1,\ldots,x_{n-1}]/(x_1,\ldots,x_{n-1})^2$:
% as considered in
%Theorem~\ref{degtheorem}.  

\begin{theorem}\label{maxrank}
Let $K$ be a field and 
$R_n=K[x_1,\dots,x_{n-1}]/(x_1,\dots,x_{n-1})^2$.  
Then for all $K$-algebras $A$ of 
dimension $n$, we have $\dim_KG({A}/K)\leq\dim_KG(R_n/K)$.
\end{theorem}
%Thus the
%dimension of the $S_n$-closure of a ring $R$ of rank $n$ over $K$ can
%be viewed as an interesting 
%measure of the ``degeneracy'' of $R$ as a ring extension
%of $K$, with $R_n$ giving the maximal possible dimension.

%What is this maximal possible dimension achieved by $R_n$?
In addition to their interest due to Theorem~\ref{maxrank}, 
the algebras $R_n$ are of interest in their own
right as they arise (with $K=\F_p$) as the reductions modulo $p$ of
orders $R$ in number fields that are {\it imprimitive} at $p$, i.e.,
$R=\Z+pR'$ for some order $R'$.  
For these reasons, we study the $S_n$-closures of these algebras in
more detail in Section~\ref{deg}, and show:

\begin{theorem}\label{degtheorem}
Let $K$ be a field of characteristic $0$ or coprime to $n!$, and let
\linebreak 
$R_n = K[x_1,\ldots,x_{n-1}]/(x_1,\dots,x_{n-1})^2$.  Then the dimension of
$G(R_n/K)$ over $K$ is strictly greater than $n!$ for $n>3$.
\end{theorem}
In particular, we find for $n=1$, 2, 3, 4, 5, and 6 that
$\dim_K\,G(R_n/K)=1$, $2$, $6$, $32$, $220$, and $1857$ 
respectively.  These ranks thus give the maximal possible ranks for the
$S_n$-closures of rank $n$ rings over $K$ for these values of $n$.  
Theorem~\ref{degtheorem}
will in fact follow from a more general structure theorem for these
rings $G(R_n/K)$ (see Theorem~\ref{thm:main}).  The techniques used to prove
Theorem~\ref{degtheorem} are primarily those of representation theory
of $S_n$.  
%Due to theorems 7 and 8, the
%dimension of the $S_n$-closure of a ring $R$ of rank $n$ over $K$ can
%be viewed as an interesting 
%measure of the ``degeneracy'' of $R$ as a ring extension
%of $K$, with $R_n$ giving the maximal possible dimension/degeneracy.

%, one finds in rank 4 that there exist algebras over fields
%for which the $S_4$-closure need not have rank 4!=24.  For instance,
%we will compute the $S_4$-closure of the ring $K[x,y,z]/(x,y,z)^2$ to
%have dimension 32, for any field $K$.  

%Theorem~\ref{degtheorem} 

As we now describe, our notion of Galois closure can also easily be adapted to
the more general situation of 
%an algebra $A$ over a ring $B$ that is
%locally free over $B$ of rank $n$.  In that case, we define 
%$G(A/B)$ by ...  More generally, we may consider a covering $\mathcal
%A/\mathcal B$ of schemes that is locally free of rank $n$
a morphism $X\to Y$ of schemes, where $\mathcal{A}$ is a locally free sheaf of $\OO_Y$-algebras of rank~$n$ and 
$X=\underline{\textrm{Spec}}_Y\mathcal{A}$.  We say then that
$X/Y$ is an {\it $n$-covering}.  

Recall that if $\mathcal{E}$ is a locally free sheaf of rank $n$ on a scheme $Y$ and $f$ is a local section of 
$\mathcal{E}nd(\mathcal{E})$, then the trace of $f$ is the image of $f$ under the canonical morphism
\[
\mathcal{E}nd(\mathcal{E})\cong \mathcal{E}\otimes_{\OO_Y}\mathcal{E}^\vee\to\OO_Y.
\]
If $X/Y$ is an $n$-covering and $\mathcal{A}$ is as above, then for
any $a\in\mathcal{A}(U)$ we can define the coefficients $s_j(a)$ of
the ``characteristic polynomial'' $P_a$ of $a$ as follows.  We obtain
an $\OO_U$-module endomorphism of $\mathcal{A}|_U$ given by
multiplication by $a$.  We let $s_j(a)$ be the trace of the induced
endomorphism of $\bigwedge^j\mathcal{A}|_U$.  We can then define a
sheaf of ideals $\mathcal{I}(\mathcal{A},\OO_Y)$ of
$\mathcal{A}^{\otimes n}$ generated by the local expressions as in
(\ref{fundrels}) and let
\[
G(\mathcal{A}/\OO_Y)=\mathcal{A}^{\otimes n}/\mathcal{I}(\mathcal{A},\OO_Y).
\]
We define 
\[
G(X/Y)=\underline{\textrm{Spec}}_YG(\mathcal{A}/\OO_Y).
\]
%The Galois
%closure $G(\mathcal A/\mathcal B)$ of an $n$-covering $\mathcal
%A/\mathcal B$ is then defined in a similar
%manner, namely ...  
Even in this more general context of $n$-coverings
of schemes, we still have the analogues of Theorems~1, 3, 4, and 5.  More precisely, \\
\\
\begin{theorem1'}
\emph{
If $X/Y$ is an $n$-covering and $Z\to Y$ is a morphism of schemes, then there is a natural isomorphism
\[
G(X/Y)\times_Y Z\cong G(X\times_Y Z / Z).
\]
}
\end{theorem1'}
\begin{theorem3'}
\emph{
If $X/Y$ is an $n$-covering defined by a locally free sheaf $\mathcal{A}$ of $\OO_Y$-algebras which is locally generated 
as an $\OO_Y$-algebra by one element, then $G(X/Y)$ is an $n!$-covering of $Y$.
}\\
\end{theorem3'}
\begin{theorem4'}
\emph{
If $X/Y$ is an $n$-covering which is \'etale, then $G(X/Y)$ is an $n!$-covering of $Y$ which is \'etale.
}\\
\end{theorem4'}
\begin{theorem5'}
\emph{
If $X/Y$ is an $n$-covering defined by a locally free sheaf $\mathcal{A}$ of $\OO_Y$-algebras and $n\leq3$, then $G(X/Y)$ 
is an $n!$-covering of $Y$.
}\\
\end{theorem5'}
%We show in Section \ref{sec:prime} how 
Theorems $1'$, $3'$, $4'$, and $5'$ follow directly from Theorems~1,
3, 4, and 5, due to the local nature of our definitions.  Hence we
will concentrate primarily 
on the proofs of Theorems 1--9, in cases
of locally free ring extensions of rank $n$.
% where $A$
%is a (free) ring of rank $n$ over $B$.

We note that the notion of $S_n$-closure considered here arises at least incidentally or in special cases in other works.  For example, it occurs in the monogenic case in~Grothendieck \cite[Lem.~1]{semchevalley} and in~Katz--Mazur \cite[\S1.8.2]{KM}.  The construction for general rings is also mentioned in \cite[\S 5.2]{Ferrand} (comment of O.\ Gabber), although no properties are proven there. 
 
We end the introduction by noting that the $S_n$-closure construction can also be characterized by a universal property in terms of a key notion of Katz and Mazur \cite[1.8.2]{KM}: if $B$ is a ring and $A$ is a $B$-algebra which is locally free of rank $n$, then $B$-algebra maps $p_1,\dots,p_n:A\to B$ form {\it a full set of sections} if for every $B$-algebra $C$ and every $f\in A\otimes_B C$,
\[
P_f(x)=\prod_{i=1}^n(x-(p_i\otimes\id)(f)).
\]
Then Theorem~\ref{thm:functorial} implies:

\begin{theorem}\label{universal}
Let $B$ be any ring and $A$ any $B$-algebra that is locally free of rank $n$.
Then $G(A/B)$ is the universal $B$-algebra over which $A$ admits a full set of $n$ sections.
\end{theorem}
Indeed, Theorem~\ref{thm:functorial} shows that the $G(A/B)$-algebra maps $p_i:A\otimes_B G(A/B) \to G(A/B)$ defined by $p_i(a\otimes \gamma)=a_i\gamma$ form a full set of sections, where $a_i$ denotes the image of $a^{(i)}$ in $G(A/B)$.  It is then immediate from the relations (\ref{fundrels}) defining $I(A,B)$ that this family is universal.

%Finally, say something about the case $k=4$ and how there is no
%subrepresentation of dimension 24.  Give an example of a nonprimitive
%order of rank 4.

%The organization of this paper is as follows.  In Section
%\ref{funcsec}, we prove Theorem \ref{thm:functorial}.  In Sections
%\ref{bnsec} and \ref{etalesec} we prove the first and second
%assertions of Theorem \ref{etalecase}, respectively.  We then handle
%Theorem \ref{fieldcase} in Section \ref{fieldsec}, Theorem
%\ref{monocase} in Section \ref{monosec}, and Theorem \ref{cubicase} in
%Section \ref{cubisec}.  In Section \ref{degexample}, we give an
%explicit computation of the $S_4$-closure of the ``maximally
%degenerate ring'' of rank 4 and show that its rank is $32>4!$.  In
%Section \ref{orderexample} we show that there is no natural (i.e.
%functorial) quotient of the $S_n$-closure whose rank is exactly $n!$.
%In Section \ref{sec:maxrank}, we prove Theorem \ref{maxrank}.  In
%Section \ref{deg} we give a structure theorem (Theorem \ref{thm:main})
%for the $S_n$-closure of the degenerate ring and use this to deduce
%Theorem \ref{degtheorem}.  We conclude by listing some open questions
%in Section \ref{sec:open}.

\section{$S_n$-closure commutes with base change}\label{funcsec}

%In this section, we prove Theorem~1.  
Let $A$ be any {ring of rank $n$ over a base ring $B$}. In this
section, we show that
%\begin{theorem}\label{addbasis}
  the ideal $I(A,B)$ in $A^{\otimes n}$ is generated by the relations
  (\ref{fundrels}), where {\em $a$ ranges over a basis of $A$ as a
    module over $B$}.  
%\end{theorem}
As such a basis 
%of $A$ as a module over $B$ will 
remains a basis of $A\otimes_B C$ as a module over $C$ for
any ring $C$, Theorem~1 will then follow.
%from the resulting equality $I(A,B)\otimes C = I(A\otimes C,B\otimes C)$.

To prove our assertion about $I(A,B)$, we require:

\begin{lemma}\label{bart}
%  Let $X$ and $Y$ be $k \times k$ matrices over a commutative ring
%  $B$.  
Let $\Z\langle X,Y\rangle$ denote the noncommutative polynomial ring
over $\Z$
generated by $X$ and~$Y$.  Then there exists a unique sequence 
  $f_0(X,Y)$, $f_1(X,Y)$, $\ldots$ of polynomials 
in $\Z\langle X,Y\rangle$ such that in $\Z\langle X,Y\rangle[[T]]$ we have:
%Let $A$, $B$ denote any two $k\times k$ matrices.  There exist unique
%polynomials $f_1,f_2,\ldots\in \Z[A,B]$ (or should we write $\Z\langle
%A,B\rangle$?) such that in $\Z[A,B][[T]]$ we have:
\begin{equation}\label{exp}
(1-(X+Y)T) = (1-XT)(1-YT)\prod_{k=0}^\infty (1-f_k(X,Y)\,XY\,T^{k+2}).
\end{equation}
Furthermore, $f_m(X,Y)$ is a homogeneous polynomial in $X$ and $Y$ of degree $m$.
\end{lemma}

\begin{proof}
%This may be proven (mod $T^i$) by induction on $i$.
We first prove by induction on $m$ that the value of $f_m(X,Y)$ is
completely determined by~(\ref{exp}).  Indeed, 
to see the assertion for $m=0$, we take (\ref{exp}) modulo
$T^3$ to obtain
\[(1-(X+Y)T) \equiv (1-XT)(1-YT)(1-f_0(X,Y)\,XY\,T^{2})\pmod{T^3}\]
implying 
\[1-XT-YT \equiv 1-XT-YT+(1-f_0(X,Y))\,XY\,T^2 \pmod{T^3}\]
and so we must have $f_0(X,Y)=1$.  

Similarly, assuming that $f_0(X,Y),\ldots,f_{m-1}(X,Y)$ have been
determined from (\ref{exp}), the polynomial $f_m(X,Y)$ can also then
be determined from (\ref{exp}) by taking (\ref{exp}) modulo $T^{m+3}$:
\begin{equation}\label{exp2}
(1-(X+Y)T)\equiv 
(1-XT)(1-YT)\prod_{k=0}^m (1-f_k(X,Y)\,XY\,T^{k+2}) \pmod{T^{m+3}};
\end{equation}
equating the coefficients of $T^{m+2}$ in (\ref{exp2}) yields 
\begin{equation}\label{exp3}
f_m(X,Y)\,XY = \Bigl[\mbox{coefficient of $T^{m+2}$ in 
$\displaystyle{
(1-XT)(1-YT)\prod_{k=0}^{m-1} (1-f_k(X,Y)\,XY\,T^{k+2})}$}\Bigr].
\end{equation}
Inspection shows that every term on the right hand side of (\ref{exp3}) is right-divisible by $XY$; dividing on the right by $XY$ on both sides of (\ref{exp3}) now gives the desired expression for $f_m(X,Y)$.

We have shown that the sequence $\{f_m(X,Y)\}$ is uniquely determined from
(\ref{exp}) via the recursive formula in (\ref{exp3}). Moreover, the equation in (\ref{exp}) is true for this latter sequence $\{f_m(X,Y)\}$ of
polynomials because it is true modulo $T^i$ for every $i$.  This
concludes the proof.
\end{proof}

\vspace{-.075in}
\begin{remark}
{\em This beautiful lemma (Lemma~\ref{bart}) 
was pointed out to us by Bart de Smit.  See also \cite{Amitsur}, \cite{RS} 
for related results.}
\end{remark}

\begin{remark}
{\em The first few polynomials $f_k(X,Y)$ are given as follows:}
\begin{equation}
\begin{array}{rcl}
f_0(X,Y)&=&1\\[.025in] 
f_1(X,Y)&=&X+Y\\[.025in]  
f_2(X,Y) &=& X^2 + YX + Y^2\\[.025in] 
f_3(X,Y) &=& X^3 + XYX + XY^2 + YX^2 +  Y^2X + Y^3\\[.025in]  
f_4(X,Y) &=& X^4 + XYX^2 + XY^2X + XY^3 + YX^3 + Y^2X^2 + Y^3X + Y^4.
\end{array}
\end{equation}
\end{remark}

We now return to our assertion about $I(A,B)$.  Given $a\in A$, let 
$Q_a(T)= \det(1-aT)=1-s_1(a)T+s_2(a)T^2 - \cdots$ be the reverse
characteristic polynomial of $a$.
%, where we use $a|_A:=\times a$ to denote the $B$-linear transformation on $A$ given by multiplication by $a$.  
Then given any elements $x,y\in A$, we have by Lemma~\ref{bart} that 
%where $a$ is again viewed as a $B$-linear transformation $\times x:A\to A$.
\begin{equation}\label{bart2}
(1-(x+y)T) = (1-xT)(1-yT)\prod_{n=0}^{m-2}(1-(f_n(x,y)\,xy)\,T^{n+2})
\pmod{T^{m+1}}.
\end{equation}
%here, for $z\in A$, we agin use $z|_A:A\to A$ to denote the $B$-linear 
%transformation on $A$ given by multiplication by $z$. 

Taking determinants of both sides of (\ref{bart2}), 
and equating powers of $T^{m}$,
yields an expression for $s_m(x+y)$ as an integer polynomial in
$s_i(x)$ ($0\leq i\leq m$), $s_i(y)$ ($0\leq
i\leq m$), and $s_i(g_j(x,y))$ ($0\leq i\leq m/2$)
for various integer polynomials $g_j$.  When $A=B^n$, 
%then $x=(x_1,\ldots,x_n)$ and $y=(y_1,\ldots,y_n)$ 
%are $n$-tuples in $B$, and 
the $s_m(z)$ (where $z=(z_1,\ldots,z_n)\in A$) become the $m$-th elementary symmetric polynomials $e_m(z_1,\ldots,z_n)$ in $z_1,\ldots,z_n$; thus our identities involving the $s_m$ turn into polynomial identities in the elementary symmetric polynomials $e_m$ in this case (indeed, since they hold with $x,y\in B^n$ for any ring $B$, they must hold identically as polynomial identities over the integers).

\begin{remark}{\em
For example, we have:}
\begin{equation*}
\begin{array}{rcl}
s_1(x+y)&\!\!=\!\!&s_1(x)+s_1(y)\\[.025in]
s_2(x+y)&\!\!=\!\!&s_2(x)+s_1(x)s_1(y)+s_2(y)-s_1(xy)\\[.025in]
s_3(x+y)&\!\!=\!\!&s_3(x) +
s_2(x)s_1(y) + s_1(x)s_2(y) + s_3(y) + s_1(xxy)+s_1(xyy) - (s_1(x) +
s_1(y)) s_1(xy).
\end{array}
\end{equation*}
\end{remark}

\vspace{.05in} Since for any $b\in B$ and $k\in\N$ we have 
$s_k(bx)=b^ks_k(x)$, it follows by induction on $m$ that the values of all
expressions of the form $s_m(a)$ ($0\leq m\leq n)$ for $a\in A$ are
determined by the values of $s_i$ ($i\leq m$) on a  basis for
$A$ as a $B$-module.  As 
the elementary symmetric polynomials $e_i$ also satisfy these same general relations as the $s_i$, 
we conclude that the ideal $I(A,B)$ in
$A^{\otimes k}$ is generated by the relations (\ref{fundrels}),
where {$a$ ranges over a $B$-basis of $A$}.  In particular, Theorem~1
follows in the case where we are considering only ring
extensions $A$ that are free of rank $n$ over~$B$.

Of course, the above argument can be modified slightly to handle the
case where $A$ is {locally} free of rank $n$ over $B$.
Indeed, in this case $A$ is still a finitely-generated 
$B$-module (see Footnote~2).
%To see this, it
%suffices to look locally; that is, we need only show that there are
%$b_i\in B$ such that $\sum b_i=1$ and $A\otimes_B B_{b_i}$ is a
%finitely-generated $B_{b_i}$-module.  This clearly holds, by the
%definition of locally free of rank $n$.  
The above argument then shows
that $I(A,B)$ is generated by the relations (\ref{fundrels}) where $a$
runs through any set of generators for $A$ as a $B$-module.  The assertion 
of Theorem~1 then follows in this generality as well.
%for rings $A$ that are locally free of rank $n$ over~$B$.

%\begin{lemma}\label{bart}
%  Let $X$ and $Y$ be $k \times k$ matrices over a commutative ring
%  $B$.  Then there exist unique polynomials
%  $f_1(X,Y),f_2(X,Y),\ldots\in\Z\langle X,Y\rangle$, independent of
%  $X$ and $Y$, such that in $\Z\langle X,Y\rangle[[T]]$ we have:
%%Let $A$, $B$ denote any two $k\times k$ matrices.  There exist unique
%%polynomials $f_1,f_2,\ldots\in \Z[A,B]$ (or should we write $\Z\langle
%%A,B\rangle$?) such that in $\Z[A,B][[T]]$ we have:
%\begin{equation}
%(1-XT)(1-YT) = \prod_{n=1}^\infty (1-f_n(X,Y)T^n).
%\end{equation}
%\end{lemma}
%
%\begin{proof}
%This may be proven by induction?
%\end{proof}
%
%\begin{remark}
%{\em This beautiful lemma (Lemma~\ref{bart}) 
%was pointed out to us by Bart de Smit.}
%\end{remark}
%For example, we have $f_1(X,Y)=X+Y$ (give a few more values?).
%
%
%Let $P(\alpha)=P(\alpha)(T)=\det(1-\alpha T|_V)\in K[T]$ and
% $P(\beta)=P(\beta)(T)=\det(1-\beta T|_V)\in K[T]$. Then
%\begin{equation}
%P(\alpha)\cdot P(\beta) = \prod_{n=1}^\infty
%P(f_n(\alpha,\beta))(T^n).
%\end{equation}
%
%Finally,
%\begin{equation}
%\prod_{i=1}^m (1-\alpha_iT)(1-\beta_iT)=\prod_{i=1}^m\,\prod_{n=1}^\infty 
%(1-f_n(\alpha,\beta)_iT^n)\in A[[T]].
%\end{equation}

%\begin{remark}
%{\em Mention relation to Brauer-Nesbitt theorem.}
%\end{remark}

\section{The case $A=B^n$}\label{bnsec}
% and the \'etale case}

%In this section we consider the case where $A=B^n$.

\subsection{A $B$-basis for $G(B^n/B)$}\label{basisforbn}

Suppose $A$ is the rank $n$ ring $B^n$ over $B$.  Let
$$e_1=(1,0,\ldots,0), \,\,e_2=(0,1,\ldots,0), \,\,\ldots\,\,,\,\,
e_n=(0,0,\ldots,1)$$ be the standard basis for $B^n$ over $B$.  
As in the introduction, for $a\in A$, we let $a^{(i)}$ denote the element $1\otimes\dots\otimes a\otimes\dots\otimes1$ 
of $A^{\otimes n}$ with $a$ in the $i$-{th} tensor factor.  
Then a natural $B$-basis for $(B^n)^{\otimes n}$ is given by
\begin{equation}\label{ebasis}
\bigl\{e_{i_1}^{(1)}e_{i_2}^{(2)}\cdots\, e_{i_n}^{(n)}\bigr\}
\end{equation}
where
$i_1,i_2,\ldots,i_n$ each range between $1$ and $n$.

We claim that a natural $B$-basis for 
$G(B^n/B)$ is also
given by (\ref{ebasis}), but where $(i_1,i_2,\ldots,i_n)$
now ranges over all {\it permutations} of $(1,2,\ldots,n)$.

To see this, we first note that any general element of the form
$e_{i_1}^{(1)}e_{i_2}^{(2)}\cdots\, e_{i_n}^{(n)}\in (B^n)^{\otimes n},$ such that
$(i_1,\ldots,i_n)$ is {\it not} 
a permutation of $(1,2,\ldots,n)$, is in fact zero in
$G(B^n/B)$.  Indeed, let $i\in\{1,\ldots,n\}$ be any element such that
$i\notin\{i_1,\ldots,i_n\}$.  Then since $\sum_{j=1}^n
e_i^{(j)}$ equals $\Tr(e_i)=1$ in $G(B^n/B)$, we deduce
$$e_{i_1}^{(1)}e_{i_2}^{(2)}\cdots e_{i_n}^{(n)} =
%1\times e_1^{(i_1)}e_2^{(i_2)}\cdots e_n^{(i_n)} =
%e_1^{(i_1)}e_2^{(i_2)}\cdots e_n^{(i_n)} \times \bigl(\sum_{j=1}^n
%e_i^{(j)}\bigr)=
\sum_{j=1}^n \bigl[e_i^{(j)}\,\cdot\, e_{i_1}^{(1)}e_{i_2}^{(2)}\cdots\,
e_{i_n}^{(n)} \bigr] = 0$$
in $G(B^n/B)$, as desired.

On the other hand, if $(i_1,\ldots,i_n)$ is a permutation of
$(1,2,\ldots,n)$, then $e_{i_1}^{(1)}e_{i_2}^{(2)}\cdots
e_{i_n}^{(n)}$ is nonzero in $G(B^n/B)$.  To prove this,  
consider the $B$-algebra homomorphism
$\phi_{(i_1,\ldots,i_n)}: (B^n)^{\otimes n} \to B$ defined by
$$\phi_{(i_1,\ldots,i_n)}\bigl(e_{i}^{(j)}\bigr)= \left\{
\begin{array}{rl}
1& \mbox{if $i=i_j$}\\
0& \mbox{otherwise}
\end{array}\right.
.$$
%$ if $i=i_j$, and $e_{i}^{(j)}\mapsto 0$
%otherwise. 
Then it is evident that the kernel of
$\phi_{(i_1,\ldots,i_n)}$ contains $I(B^n,B)$, so that $\phi$ descends
to a map 
\[
\bar\phi_{(i_1,\ldots,i_n)}: G(B^n/B) \to B.
\]
Moreover, 
we have $\bar\phi_{(i_1,\ldots,i_n)}\bigl(e_{i_1}^{(1)}e_{i_2}^{(2)}\cdots\,
e_{i_n}^{(n)}\bigr)=1$.  We conclude that 
$e_{i_1}^{(1)}e_{i_2}^{(2)}\cdots e_{i_n}^{(n)}$ is nonzero in
$G(B^n/B)$. 

Finally, note that $e_{i_1}^{(1)}e_{i_2}^{(2)}\cdots\, e_{i_n}^{(n)}$
is an idempotent for any permutation $(i_1,\ldots,i_n)$, 
and if $(j_1,\ldots,j_n)$ is any other permutation of 
$(1,\ldots,n)$, then 
$$e_{i_1}^{(1)}e_{i_2}^{(2)}\cdots\, e_{i_n}^{(n)}
\,\cdot\, e_{j_1}^{(1)}e_{j_2}^{(2)}\cdots\, e_{j_n}^{(n)}=0.$$  Hence the
set (\ref{ebasis}), where $(i_1,i_2,\ldots,i_n)$
ranges over all {\it permutations} of $(1,2,\ldots,n)$, forms a
set of nonzero orthogonal idempotents that spans $G(B^n/B)$ as a
$B$-module. We conclude that it forms a basis for $G(B^n/B)$, as claimed.

Finally, since this basis for $G(B^n/B)$ 
has $n!$ elements, and consists entirely 
of idempotents, we conclude that $G(B^n/B)\cong B^{n!}$ as
$B$-algebras, as desired.  

We have proven the first assertion of Theorem \ref{etalecase}.

\subsection{The action of $S_n$ on $G(B^n/B)$}\label{snaction}

It is interesting to consider the natural action of $S_n$ on
$(B^n)^{\otimes n}$, and on $G(B^n/B)$, obtained by permuting the
tensor factors.  From this point of view, we see that 
\[
G(B^n/B)\cong B[S_n]
\]
as $B[S_n]$-modules.  The isomorphism is given by
$e_{i_1}^{(1)}e_{i_2}^{(2)}\cdots\, e_{i_n}^{(n)}\mapsto\sigma$, where
$\sigma\in S_n$ denotes the permutation $j\mapsto i_j$.  If we write
$e_\sigma:=e_{i_1}^{(1)}e_{i_2}^{(2)}\cdots\, e_{i_n}^{(n)}$, then the
action of an element $g\in S_n$ on $G(B^n/B)$ is given by
\[g\cdot e_\sigma = e_{g\sigma}.\]

Let $A=B^n$.  Under the action of $S_n$ on $G(A/B)$, the ring
$A^{(1)}\subset G(A/B)$ given by the image of $A\otimes 1\otimes
\cdots\otimes 1$ is fixed by the group $S_{n-1}^{(1)}$, the
subgroup of $S_n$ fixing 1.  Note that $A^{(1)}\cong A$.  
Similarly, as in Galois theory, 
the other ``conjugate'' copies of
$A$ in $G(A/B)$, namely $A^{(j)}=1\otimes\cdots\otimes
A\otimes\cdots\otimes 1$ (where the $A$ is in the $j$-th tensor
factor) for $j=2,\ldots, n$ are fixed by the conjugate subgroups
$S_{n-1}^{(j)}\subset S_n$ fixing $j$ for $j=2,\ldots, n$, respectively.  

In terms of these subgroups $S_{n-1}^{(j)}\subset S_n$, we may express the
idempotents $e_i^{(j)}$ in terms of our orthogonal 
basis $\{e_\sigma\}_{\sigma\in
  S_n}$ of idempotents for $G(A/B)$ as follows:
\begin{equation}\label{eij}
e_i^{(j)} = \sum_{\sigma\in S_{n-1}^{(j)}g_{ji}} e_\sigma,
\end{equation}
where $g_{ji}$ denotes any element in $S_n$ taking $i$ to $j$.  That
is, $e_i^{(j)}$ corresponds to the sum of $e_\sigma$ over a right coset
of $S_{n-1}^{(j)}$, namely, the right coset consisting of elements in $S_n$ taking
$i$ to $j$.  
%In particular, we again see from the identity (\ref{eij})
%that $e_i^{(j)}\subset A^{(j)}$ is indeed fixed under the left action
%of $S_{n-1}^{(j)}$.
%Our expression 
%This expression for $e_i^{(j)}$ will be useful to us when we consider
%the case where $A/B$ is a separable extension of fields.

%\subsection{The action of $S_n$ on $G(B^n/B)$ with $G$-structure}\label{snactiong}
%
%To make a full analogy with Galois theory, we consider a transitive
%subgroup $G$ of $S_n$.  We wish to constructz
For an extension of these results to general products of rings $A=A_1\times\cdots\times A_k$, 
see Section~\ref{directsumsec}.

\section{The case of fields}\label{fieldsec}

Before proving Theorem \ref{fieldcase}, we begin by recalling the correspondence between 
finite \'etale extensions of a field and Galois sets.  
Let $K$ be a field and fix a separable closure $\bar{K}$ of $K$.  Then, given a finite \'etale 
extension $L/K$, consider the set $S_{L/K}$ of $K$-algebra homomorphisms from $L$ to $\bar{K}$.  
We see that $G_K:=\textrm{Gal}(\bar{K}/K)$ acts on $S_{L/K}$ by composition: 
if $\tau\in G_K$ and $\psi\in S_{L/K}$, then $\tau\circ\psi\in S_{L/K}$.  Moreover, this action is 
continuous when $G_K$ is given the profinite topology and $S_{L/K}$ is given the discrete 
topology, i.e., the action of $G_K$ factors through a finite quotient of $G_K$.  We therefore 
obtain a functor
\begin{eqnarray}
\label{eqn:equiv}
(\textrm{finite\ \'etale\ }K\textrm{-algebras})\longrightarrow (\textrm{finite\ sets\ with\ continuous\ }G_K\textrm{-action})
\end{eqnarray}
sending $L$ to $S_{L/K}$, which is in fact an equivalence of categories (see, e.g., \cite[Thm.~2.9]{Lenstra}).

Note that if $L/K$ is finite \'etale of degree $n$, then $\bar{K}\otimes_K L$ is isomorphic to $\bar{K}^n$ as a 
$\bar{K}$-algebra.  More canonically, we have an isomorphism
\begin{eqnarray}
\label{eqn:splitting}
\bar{K}\otimes_K L \longrightarrow \bar{K}^{S_{L/K}}:=\prod_{s\in S_{L/K}}\bar{K}\nonumber\\
1\otimes\ell \longmapsto (s(\ell))_{s\in S_{L/K}}
\end{eqnarray}
of $\bar{K}$-algebras.  The Galois group $G_K$ acts on $\bar{K}\otimes_K L$ through the left tensor factor, and 
therefore induces an action on $S_{L/K}$ via (\ref{eqn:splitting}); this is precisely the $G_K$-action on 
$S_{L/K}$ in (\ref{eqn:equiv}).

We now turn to the problem of describing the Galois set $S_{G(L/K)/K}$ in terms of the Galois set 
$S_{L/K}$, where $L/K$ is a finite \'etale extension of degree $n$.  
%We index the $K$-algebra homomorphisms from $L$ to $\bar{K}$ by $1,\dots,n$.  This yields an identification of $S_{L/K}$ with $\{1,2,\dots,n\}$.  
By Theorem \ref{thm:functorial},
\[
\bar{K}\otimes_K G(L/K)\cong (\bar{K}^{S_{L/K}})^{\otimes{n}}/I(\bar{K}^{S_{L/K}},\bar{K})
\]
as $\bar{K}$-algebras.  The $G_K$-action on the left tensor factor of $\bar{K}\otimes_K G(L/K)$ yields an 
action on $(\bar{K}^{S_{L/K}})^{\otimes{n}}/I(\bar{K}^{S_{L/K}},\bar{K})$ defined by
\[
\tau(e_s^{(i)})=e_{\tau(s)}^{(i)},
\]
where $\tau\in G_K$, $s\in S_{L/K}$, and $e_s$ denotes the element $(\delta_{ss'})_{s'}$ of $K^{S_{L/K}}$, where $\delta$ is the Kronecker delta 
function.

Given two sets $T$ and $T'$, we let $\bij(T,T')$ denote the set of bijections $f:T\to T'$.  Let $[n]$ denote the set $\{1,2,\dots,n\}$.  Then Section \ref{snaction} shows that
\[
\bar{K}[\bij([n],S_{L/K})]\stackrel{\cong}{\longrightarrow}(\bar{K}^{S_{L/K}})^{\otimes{n}}/I(\bar{K}^{S_{L/K}},\bar{K})
\]
as $\bar{K}$-algebras, where $\pi\in \bij([n],S_{L/K})$ is sent to 
$e_{\pi(1)}^{(1)}e_{\pi(2)}^{(2)}\dots e_{\pi(n)}^{(n)}$.  We see then that the action of $G_K$ on 
$\bar{K}\otimes_K G(L/K)$ induces an action on $\bij([n],S_{L/K})$ where $G_K$ acts on $\bij([n],S_{L/K})$ via its action on $S_{L/K}$; that is, 
$\tau\in G_K$ acts on $\pi\in\bij([n],S_{L/K})$ by
\[
(\tau(\pi))(j)=\tau(\pi(j)).
\]
Hence, the Galois set corresponding to $G(L/K)$ is given by $\bij([n],S_{L/K})$ with this action of $G_K$. 
%More canonically, 
%\begin{eqnarray}
%\label{eqn:canGalois}
%S_{G(L/K)}=\textrm{Perm}(S_{L/K})
%\end{eqnarray}
%as sets and the $G_K$-action on $S_{G(L/K)}$ is given by 
%\begin{eqnarray}
%\label{eqn:canGalois2}
%(\tau(f))(s)=\tau(f(s)),
%\end{eqnarray}
%where $\tau\in G_K$, $f\in \textrm{Perm}(S_{L/K})$, and $s\in S_{L/K}$.

We now prove Theorem \ref{fieldcase}.  Let $L$ be a finite separable 
field extension of $K$ of degree $n$, and let $M$
be the Galois closure of $L/K$ in $\bar{K}$.  Let $G$ denote the Galois group of
$M/K$; thus $G$ acts faithfully and transitively on $S_{L/K}$, 
% elements, namely, on the $n$ embeddings of $L$ into $M$, which we index by $1,\ldots,n$.  
%Let $|G|=m$, and let $r=n!/m$.  
%, where $S_n$ denotes the group of permutations on the set $\{1,\ldots,n\}$.  Using (\ref{eqn:canGalois}), we show $G(L/K)\cong M^r$ as $K$-algebras.
and so $G$ sits naturally inside $\bij(S_{L/K},S_{L/K})$.  Note that $\bij(S_{L/K},S_{L/K})$ carries a $G_K$-action via post-composition.  
%Our indexing of the embeddings of $L$ into $M\subset\bar{K}$ identifies $S_{L/K}$ with $\{1,\dots,n\}$.  We can then define an action of $G_K$ on  $S_n$ by $(\tau(\pi))(j)=\tau(\pi(j))$, where $\tau\in G_K$ and $\pi\in S_n$; note that this yields the same Galois set as that for $G(L/K)$.  
Since $M$ is the Galois closure of $L/K$, this action of $G_K$ 
%on $S_n$ 
restricts to an action on $G$.  
%\subset S_n$.  
The set $G$ equipped with this action is the Galois set corresponding to $M$.  Since the action of $G$ on $S_{L/K}$ is faithful and transitive, we see that as Galois sets, 
\[
\bij(S_{L/K},S_{L/K})=\coprod_r G=S_{M^r/K}.
\]
%Now $M^r$ corresponds to the disjoint union of $r$ copies of this Galois set.  As sets, $S_{M^r}$ is of course in bijection with $S_n=S_{G(L/K)}$; what we must show is that there is a $G_K$-equivariant bijection.  Writing $S_{M^r}$ as 
%\[\coprod_{a\in G \backslash S_n}G\cdot a,\]
%we see that the  action of $\tau\in G_K$ is given by $\tau(ga)=\tau(g)a$, where $g\in G$.  Note that this agrees with the action of $G_K$ on $S_n$ defined by (\ref{eqn:canGalois2}).
Therefore, choosing a bijection of $[n]$ and $S_{L/K}$ yields an isomorphism $S_{G(L/K)/K}\cong S_{M^r/K}$ of Galois sets.  
As a result, 
%$S_{G(L/K)}$ and $S_{M^r}$ are isomorphic as Galois sets; hence 
$G(L/K)$ and $M^r$ are isomorphic as $K$-algebras.

\section{The \'etale case}\label{etalesec}

We have already proven the first assertion of Theorem~\ref{etalecase}.
Suppose, more generally, that $A$ is any ring that is \'etale and
locally free of rank $n$ over $B$.  Then we claim that $G(A/B)$ is an
\'etale $B$-algebra which is locally free of rank $n!$; this is the second
assertion of Theorem~\ref{etalecase}.

%\begin{proposition}
%If $B$ is an \'etale $A$-algebra which is free of rank $n$ as an $A$-module, then $G(B/A)$ is an \'etale $A$-algebra 
%of rank $n$!.
%\end{proposition}
%\begin{proof}
%If $X$ and $Y$ are non-empty schemes and $f:X\to Y$ is a finite \'etale morphism of degree $n$, then there  
%exists an \'etale cover $Z\to Y$ such that $X\times_Z Y\to Z$ is the trivial \'etale cover (see 
%for example \cite[p.26]{milne}).  By functoriality of the $S_n$-closure, the proposition can be proved \'etale locally 
%on $A$, where we may therefore assume that $B=A^n$.  The result then follows from ??.
%\end{proof}
%\begin{proof}
To prove the claim, we first require a definition.  An \'etale
$B$-algebra $C$ is called an {\it \'etale cover} of $B$ if the induced
morphism $\Spec C\to\Spec B$ is surjective.  
%(see, e.g., \cite[p.\ 47]{milne}).  
The key fact we
use in proving the second assertion of Theorem~\ref{etalecase} is
the following well-known lemma:
% (see, for example, \cite[p.\ 156]{milne} or \cite[Thm.~5.10]{Lenstra}):
%proof of \cite[Prop 6.16]{milne}) 

\begin{lemma}\label{split}
  Let $R$ be any $B$-algebra that is finitely generated as a
  $B$-module.  Then $R$ is \'etale and locally free of rank $n$ over
  $B$ if and only if there exists an \'etale cover $C$ of $B$ such
  that $R\otimes_B C\cong C^n$ as $C$-algebras.
\end{lemma}

\begin{proof}
First, the proof of \cite[Thm.~11.4]{milne} shows that $R$ is locally free of rank $n$ over $B$ if and only if there exists an \'etale cover $C$ of $B$ such that $R\otimes_B C\cong C^n$ as $C$-modules.

Now if $C$ is an \'etale cover of $B$ as in the statement of the lemma, then $R\otimes_B C$ is an \'etale $C$-algebra.  By \'etale descent, $R$ is then an \'etale $B$-algebra (see, e.g.,  \cite[Prop.~1.15(x)]{vistoli}), and it is also locally free of rank $n$ over $B$ by the previous paragraph.

Conversely, if $R$ is \'etale and locally free of rank $n$ over $B$, then $R$ is an \'etale cover of $B$.  There then exists an $R$-algebra $R'$ such that we have an isomorphism of $R$-algebras $R\otimes_B R\cong R\times R'$ (see, e.g., \cite[Remark after Lemma~1.1.17]{masdeu}).  Since $R\otimes_B R$ is \'etale and locally free of rank $n$ over $R$, it follows that $R'$ is \'etale and locally free of rank $n-1$ over $R$.  By induction on $n$, we conclude that there exists an \'etale cover $C$ of $B$ such that $R\otimes_B C\cong C^n$ as $C$-algebras.
%
%First, the proof of \cite[Thm.\ 11.4]{milne} shows that $R$ is locally free of rank $n$ over $B$ if and only if there exists an \'etale cover $C$ of $B$ such that $R\otimes_B C\cong C^n$ as $C$-modules.
%
%If $C$ is an \'etale cover of $B$ as in the statement of the lemma, then $R\otimes_B C$ is an \'etale $C$-algebra.  It follows from \cite[Prop.\ 1.15(x)]{vistoli} that $R$ is an \'etale $B$-algebra.  It is locally free of rank $n$ by the previous paragraph.
%
%Conversely, if $R$ is \'etale and locally free of rank $n$ over $B$, then $R$ is an \'etale cover of $B$.  By \cite[Cor.\ 3.12]{milne}, there is an $R$-algebra $R'$ such that $R\otimes_B R\cong R\oplus R'$ as $R$-algebras.  Since $R\otimes_B R$ is \'etale and locally free of rank $n$ over $R$, it follows that $R'$ is \'etale and locally free of rank $n-1$ over $R$.  Hence, by induction on $n$, there exists an \'etale cover $C$ of $B$ such that $R\otimes_B C\cong C^n$ as $C$-algebras.
\end{proof}

%that a $B$-algebra $R$ which is finitely generated as a $B$-module is
%an \'etale $B$-algebra if and only if there exists an \'etale cover
%$C$ of $B$ such that
%$R\otimes_B C\cong C^n$ as a $C$-algebra for some $n$;
%see, for example, 
%%the paragraph before Exercise 1.3 in 
%\cite[V \S1,p.?]{milne}.

%{}Using this key fact, 
Since $A$ is \'etale and locally free of rank $n$ over $B$, by
Lemma~\ref{split} we see that there exists an \'etale cover $C$ of $B$
such that $A\otimes_B C\cong C^n$ as $C$-algebras.  By Theorem \ref{thm:functorial}, 
we then have
\[
G(A/B)\otimes_{B}C\cong G(C^n/C)\cong C^{n!}
\]
as $C$-algebras.  Applying Lemma~\ref{split} once again, we conclude that
$G(A/B)$ is \'etale and locally free of rank $n!$ over $B$, as desired.

We can say more in terms of the underlying Galois sets when $\Spec B$ is connected.  
Recall that there is an equivalence of categories between finite \'etale extensions of $B$ and finite sets equipped 
with a continuous action by a certain profinite group $\pi_1^{\textrm{\'et}}(B)$ called the 
\emph{\'etale fundamental group} of $B$ (see, e.g., \cite[Thm.~1.11]{Lenstra}).  When $B=K$ is a field, $\pi_1^{\textrm{\'et}}(K)$ is nothing 
other than $G_K$.  By the same argument as in the case of fields, 
one shows that if $A/B$ is a finite \'etale extension of degree~$n$ corresponding to a finite set $S$ with a continuous action 
by $\pi_1^{\textrm{\'et}}(B)$, then $G(A/B)$ corresponds to the set $\bij([n],S)$ with 
$\pi_1^{\textrm{\'et}}(B)$-action induced by the action on $S$.
%\[(\tau(f))(s)=\tau(f(s)),\]
%where $\tau\in \pi_1^{\textrm{\'et}}(B)$, $f\in \textrm{Perm}(S)$, and $s\in S$.

%To see this, it suffices to find an \'etale cover $\Spec C\to\Spec B$ (i.e. an \'etale 
%$B$-algebra $C$ such that the induced morphism $\Spec C\to\Spec B$ is surjective) with $G(A/B)\otimes_{B}C$ 
%an \'etale $C$-algebra.  By \cite[p.26]{milne}, there is an \'etale cover $\Spec C\to\Spec B$ such that 
%$A\otimes_{B}C=C^n$.  (Explicitly, $C$ can be taken to be $A^{\otimes n}$.)  With this choice of $C$, functoriality of 
%the $S_n$-closure shows that 
%\[
%G(A/B)\otimes_{B}C=G(C^n/C)=C^{n!},
%\]
%and so $G(A/B)$ is \'etale of rank $n!$.
%\end{proof}

%\begin{proposition}\label{bn} 
%A $B$-basis for $G(B^n/B)=(B^n)^{\otimes n}/I(B^n,B)$ is given by
%$$\{e_1^{(i_1)}e_2^{(i_2)}\cdots e_n^{(i_n)}\}$$ where
%$(i_1,i_2,\ldots,i_n)$ ranges over all permuations of $1,2,\ldots,n$.
%\end{proposition}

%\begin{proof}
%
%\end{proof}

\section{The monogenic case}\label{monosec}

In this section, we examine the situation where $B$ is monogenic over
$A$.  We prove:

\begin{theorem}\label{mon}
  Let $f$ be a monic polynomial with coeffiecients in $B$, and let 
$A=B[x]/f(x)$ denote the corresponding monogenic ring of rank $n$ over $B$. 
%where $x\in A$ satisfies a degree $n$ polynomial
%  equation over $B$.  
Then the ring $G(A/B)$ is a ring of rank $n!$ over
  $B$, a basis of it being monomials of the form
\begin{equation}\label{monobasis}
\prod_{i=1}^n x_i^{e_i}
\end{equation}
where the exponents $e_i$ satisfy $0\leq  e_i < i$; here
$x_1,x_2,\ldots,x_n$ denote the images in $G(A/B)$ of 
%  $x\otimes \ldots\otimes1\otimes 1$, \,$1\otimes x\otimes \ldots\otimes
 % 1$, \,$\ldots$\,, \,$1\otimes
 % 1\otimes \ldots\otimes x
$x^{(1)}, x^{(2)}, \ldots, x^{(n)}   
  \in A^{\otimes n}$ respectively.
 %, as defined in the introduction.
\end{theorem}

\begin{proof}
If $A=B[x]$, then $I(A,B)$ is generated by the relations
(\ref{fundrels}) where $a=x$.  This is because the powers
$1,x,x^2,\ldots,x^{n-1}$ of $x$ form a basis for $A$ over $B$, and the
elementary symmetric functions $s_i(x^j)$ of powers $x^j$ of $x$ are
integer polynomials in the elementary symmetric functions $s_i(x)$ of
$x$ (by the {\it Newton--Girard identities}; see, e.g.,~\cite[pp.\ 99--101]{vdw}).  
Hence the relations (\ref{fundrels}) for $a=x^j$ ($j>1$) 
are implied by those for which $a=x$.

  Now let the characteristic polynomial of $\times x:A\to A$ be given by
  $P_x(T)=T^n-s_1(x)T^{n-1}+s_2(x) T^{n-2}+\cdots+(-1)^n s_n(x)$.
  Then a direct construction of $G(A/B)$ is as follows.  By the
  symmetric function theorem, the ring $R = \Z[X_1,...,X_n]$ is a free
  module of rank $n!$ (with basis given as above) over the polynomial
  ring $S = \Z[\Sigma_1,...,\Sigma_n]$, where the $\Sigma_i$ 
denote the elementary
  symmetric polynomials $\Sigma_1 = X_1 + ... + X_n$, etc. Using the
  coefficients of $P_x$, we get a map $\psi:S\to B$ defined by sending
  $\Sigma_i$ to $s_i(x)$.  This
  allows us to construct the $B$-algebra $R\otimes_S B$, which is
  then free over $B$ of rank $n!$. 

  We claim that the algebra $R\otimes_S B$ is isomorphic to $G(A/B)$.
Indeed, we may define a map $$\phi:A^{\otimes n}=B[x^{(1)},\ldots,x^{(n)}]\to
R\otimes_S B$$ by sending $x^{(i)}\mapsto X_i\otimes 1$.  Then, by the definition of 
$R\otimes_S B$, the kernel of $\phi$ consists of all polynomials in
$B[x^{(1)},\ldots,x^{(n)}]$ that are symmetric in the $x^{(i)}$ (so can be expressed as polynomials in the elementary symmetric functions $e_j(x^{(1)},\ldots,x^{(n)})$)
and that evaluate to 0 when $e_j(x^{(1)},\ldots,x^{(n)})$ is replaced by $s_j(x)$ for every $j$.  But
these polynomials are precisely the elements of the 
ideal $I(A,B)$, and thus
$G(A/B)=A^{\otimes n}/I(A,B)\cong R\otimes_SB$, as desired.
 %
 % This is shown by constructing maps in both directions: $R\otimes_S B
 % \to G(A/B)$ (sending $X_i\mapsto x_i$) and $G(A/B) \to R\otimes_S B$
%(sending $x_i\mapsto X_i$).
\end{proof}

%In particular, we have proven Theorem~\ref{monocase}.
\vspace{-.1in}
\begin{remark}
{\em This construction in the monogenic case is more or less given in Grothendieck~\cite[Lemme 1 and corollary next page]{semchevalley}}.
\end{remark}

In the case that $A$ is monogenic over the base ring $B$, we may use
Theorem~\ref{mon} to compute the discriminant $\Disc(G(A/B))$ of the
$S_n$-closure of $A$ in terms of the discriminant $\Disc(A)$ of $A$.  
The definition of $\Disc(A)\in B$ as given in the introduction (prior to Theorem~\ref{etalecase}) is only well-defined up to factors in $B^{\times2}$.  However, if $A$ is a {\it based ring} of rank $n$ over $B$---i.e., $A$ comes equipped with a basis of size $n$ over $B$---then $\Disc(A/B)$ is well-defined as an element of $B$, namely as the determinant of the trace form $\Tr(xy)$ on $A$ expressed as an $n\times n$ matrix over $B$ in terms of this chosen basis. We then
find that, for $n\geq 2$, we have
\begin{equation}\label{discid}
\Disc(G(A/B))=\Disc(A)^{n!/2}
\end{equation}
as discriminants of based rings over $B$; here $A$ is equipped with its power basis and $G(A/B)$ is given the basis as in Theorem~\ref{mon}.
 
To see this, note that it suffices again to prove this identity in
the case $B=\Z[\Sigma_1,\ldots,\Sigma_n]$,
$A=B[X_1]$, and $G(A/B) = \Z[X_1,...,X_n]$.  The identity
(\ref{discid}) is trivial for $n=2$, while for general~$n$ it follows
by induction.  Indeed, we have the equalities 
$A=\Z[X_1][\Sigma'_1,\ldots,\Sigma'_{n-1}]$ and $G(A/B)=A[X_2,\ldots,X_n]$, where
$\Sigma'_1,\ldots,\Sigma'_{n-1}$ denote the elementary symmetric
polynomials in $X_2,\ldots,X_n$.  
The proof of Theorem~\ref{mon} now implies that
$G(A/B)=G(A[X_2]/A)$, so that $G(A/B)$ is free of rank $(n-1)!$
over~$A$. The induction hypothesis then gives 
$$\Disc(G(A/B)/A)=\Disc(A[X_2]/A)^{(n-1)!/2}.$$  
In the tower of ring extensions $G(A/B)\,/\,A\,/\,B$, we then see that
\[\begin{array}{rcl}
\Disc(G(A/B)) &=& {\mathbf N}^A_B(\Disc(G(A/B)/A))\cdot
\Disc(A)^{(n-1)!} \\[.1in] &=& \Disc(A)^{(n-2)\cdot(n-1)!/2}\cdot \Disc(A)^{(n-1)!}\\[.1in]
&=&\Disc(A)^{n!/2}, 
\end{array}\]
proving (\ref{discid}).

%\begin{equation}
%\begin{array}{c}
%G(A/B)=G(A[X_2]/A)\\
%|\\
%A\\
%|\\
%B
%\end{array}
%\end{equation}

\section{Ranks $k\leq 3$}
\label{cubisec}

The cases $k=1,2$ in Theorem~\ref{cubicase} follow from
Theorem~\ref{monocase}.  So we consider the case $k=3$ in this
section.  
%We start by considering the where $A$ is
%free over $B$ with basis of the form $1$, $x$, $y$.

\begin{theorem}\label{cubicase2}
  Assume that $A$ is free of rank $3$ over $B$ with basis $1$, $x$, $y$.
  Let $x_1$, $x_2$, $x_3$ and $y_1$,~$y_2$, $y_3$ 
  denote the images in $G(A/B)$ of the
  elements 
%  $x\otimes 1\otimes 1$, $1\otimes x\otimes 1$, $1\otimes
 % 1\otimes x\in A^{\otimes 3}$ 
$x^ {(1)}$, $x^{(2)}$, $x^{(3)}$  and $y^{(1)}$, $y^{(2)}$, $y^{(3)}\in A^{\otimes 3}$ 
  respectively.
%  , and define
 % $y_1,y_2,y_3$ similarly.  
 Then the ring $G(A/B)$ is free of rank $6$
  over $B$ with basis $1, x_1, y_1, x_2, y_2, x_1y_2.$
%In particular, if $B = \Z$, then $G(A/B)$ has no additive torsion.
\end{theorem}  

\begin{proof}
It is known that, by translating $x$ and $y$ by appropriate
$B$-multiples of 1, the multiplication table of $A$ as a ring
over $B$ can be expressed in the form
\begin{equation}\label{ringlaw3}
\begin{array}{cll}
xy &=& \,\,\,\,\;\!ad \\
x^2 &=& -ac + b x \;\!+ a y \\
y^2 &=& -bd\;\! + d x \;\!\!+ c y.
\end{array}
\end{equation}
for some elements $a,b,c,d\in B$ (see \cite[Prop.~4.2]{GGS}).

In terms of these elements, the characteristic
equations of $x$, $y$, and $x+y$ are given by
$$T^3- bT^2 + acT -a^2d=0,$$ $$T^3- cT^2 + bdT -ad^2 =0,$$
and
$$T^3- (b + c)T^2 +
(ac + bc +bd - 3ad)T - (
a^2d +ac^2 + b^2d + ad^2- 2abd - 2acd )=0
$$ 
respectively.

%In terms of these elements, the characteristic
%equations of $x$ and $y$ are given by
%$$T^3- bT^2 + acT -a^2d=0$$ and $$T^3- cT^2 + bdT -ad^2 =0$$ respectively.
%T^3- (bs + ct)T^2 
%+ (acs^2 + bcst - 3adst + bdt^2)T 
%-a^2ds^3 - ac^2s^2t + 2abds^2t - b^2dst^2 + 2acdst^2 - ad^2t^3 

Note first that the trace relations $x_1 + x_2 + x_3 = b$ and $y_1 + y_2 +
y_3 = c$ are equivalent to 
\begin{eqnarray}\label{x3}
x_3&=&b-x_1-x_2,\\  \label{y3}
y_3&=&c-y_1-y_2. 
\end{eqnarray}
%can be obtained from $x_1,x_2,y_1,y_2$. 
Hence $G(A/B)$ is generated as a
$B$-module by the 9 elements $\{1,x_1,y_1\}\cdot\{1,x_2,y_2\}$, and we
need to find 3 additive relations to relate $x_1x_2$, $x_1y_2$,
$x_2y_1$, $y_1y_2$.  

Since all other trace relations are $B$-linear combinations of the
trace relations for $x$ and $y$, they do not yield any further new
relations.  Instead, we now take the quadratic identities $x_1x_2 +
x_1x_3 + x_2x_3 = ac$, $y_1y_2 + y_1y_3 + y_2y_3 = bd$, and
$(x_1+y_1)(x_2+y_2) + (x_1+y_1)(x_3+y_3) + (x_2+y_2)(x_3+y_3)
=ac + bc +bd - 3ad$, which reduce to:
\begin{eqnarray}\label{x1x2}
x_1x_2&=&a(c-y_1-y_2),\\  y_1y_2&=&d(b-x_1-x_2) \label{y1y2} \\ 
y_1 x_2 &=& bc -ad - b(c-y_1-y_2) - c(b- x_1-x_2)-x_1 y_2. \label{y1x2}
\end{eqnarray}
These identities show that $G(A/B)$ is spanned over $B$ by the six elements
claimed in the theorem.

It remains to show that these six elements are in fact linearly
independent.  By Theorem~\ref{thm:functorial}, it suffices to consider the case when
$B=\Z[a,b,c,d]$ is a free polynomial ring over $\Z$ in variables $a,b,c,d$, and
$A$ is free of rank 3 over $B$ with basis 1, $x$, $y$, and
multiplication table given by (\ref{ringlaw3}).  

In that case, let $K$ be the quotient field of $B$.  If the six
elements 1, $x_1$, $x_2$, $y_1$, $y_2$, $x_1y_2$ satisfy a linear
relation over $B$, then they also satisfy a linear relation over $K$.
We show that this is not the case.  Since $\Disc((A\otimes_B K)/K)$ is a
non-zero polynomial in $a$, $b$, $c$, and $d$, it is invertible in $K$
and hence $(A\otimes_B K)/K$ is \'etale.  In fact, $A\otimes_B K$ is a
field.  If it were not, then the cubic polynomial $f(x)$ defining the
extension $(A\otimes_B K)/K$ would have a root in $K$.  As $A/B$ is the
universal based cubic ring extension having first basis element 1,
%(via specialization of the variables $a,b,c,d$), 
this would imply that every cubic polynomial over
$\Q$ has a rational root, which is not true.

Now, by Theorem \ref{thm:functorial}, the elements 1, $x_1$, $x_2$, $y_1$, $y_2$,
$x_1y_2$ also span $G((A\otimes_B K) / K)\cong G(A/B)\otimes_B K$.
Since $A\otimes_B K$ is a field, Theorem~\ref{fieldcase} implies that
$G((A\otimes_B K) / K)$ is a $6$-dimensional vector space over $K$.  It
follows that 1, $x_1$, $x_2$, $y_1$, $y_2$, and $x_1y_2$ are linearly
independent over $K$, and hence over $B$, as desired.
\end{proof}

Thus to {any} cubic ring $A$ over $B$ with basis $1$, $x$, $y$, there
is naturally associated a canonical sextic ring $\tilde A$ over $B$, 
given by $G(A/B)$.  
We show that in fact we have the formula
\begin{equation}\label{discid2}
\Disc(\tilde A)=\Disc(A)^3\end{equation}
as discriminants of based rings.

To see this, it again suffices to
check this in the case that the base ring $B$ is $\Z[a,b,c,d]$.  In this
case, it is clear that the multiplication table for
$\tilde A$, in terms of our chosen basis $1$, $x_1$, $x_2$, $y_1$,
$y_2$, $x_1y_2$ for $\tilde A$, will involve only polynomials in
$a,b,c,d$ with coefficients in $\Z$.  Thus the discriminant
$\Disc(\tilde A)$ of $\tilde A$ will also be an integer polynomial in
$a,b,c,d$.  Furthermore, this polynomial $\Disc(\tilde A)$ must remain
invariant under changes of the basis $x,y$ via transformations in $\GL_2(\Z)$, 
which changes $a,b,c,d$ by the action of $\GL_2(\Z)$ on the binary
cubic form $f(x,y)=ax^3+bx^2y+cxy^2+dy^3$ (by \cite[Prop.~4.2]{GGS}).  
It is known (see, e.g.,
\cite[Lec.~XVII]{Hilbert}) that the only
$\GL_2(\Z)$-invariant polynomials in $a,b,c,d$ under this action must
be polynonomials in $\Disc(f)=\Disc(A)$, and thus $\Disc(\tilde A)$
must be a polynomial in $\Disc(A)$.  

To determine this polynomial, we may then restrict to the case where
$a=1$ or $d=1$; that is, we may assume 
the rank 3 ring $A$ is monogenic over $B$, in which case $\Disc(\tilde
A)=\Disc(A)^3$ by (\ref{discid}).  Formula (\ref{discid2}) 
therefore follows for general rank 3 rings $A$ over $B$ that have a basis of the form $1,x,y$. 

\vspace{.1in}
In particular, if $A$ is a cubic order in a noncyclic cubic field $K$,
then $G(A/\Z)$ provides a canonically associated sextic order $\tilde A$
in the Galois closure $\tilde K$ satisfying $\Disc(\tilde
A)=\Disc(A)^3$. 

\vspace{.1in} We may now deduce the more general
Theorem~\ref{cubicase} from Theorem~\ref{cubicase2}.  Indeed, let $A$
be any locally free ring of rank 3 over $B$.  Then it follows
from~\cite[Lemma~1.1]{GL} that, for any maximal ideal $M$ of $B$, the
localization $A_M$ is free of rank 3 over $B_M$ {\it with a basis of
  the form $1$, $x$, $y$} (essentially an application of Nakayama's
Lemma).  We conclude then, by Theorem~\ref{cubicase2}, that the
localization $G(A/B)_M$ is
free of rank~6 over $B_M$, for all maximal ideals $M$ of $B$.

Since $A$ is finitely presented as a $B$-module (being locally
free; see Footnote~2 on p.\ 1) and the ideal $I(A,B)$ is finitely
generated (a set of generators being the relations (\ref{fundrels}),
where $a$ ranges over a spanning set for $A$ over $B$; see Section~2),
we conclude that $G(A/B)$ too is finitely presented as a $B$-module.  

Finally, since $G(A/B)$ is finitely presented as a $B$-module, and the
localization $G(A/B)_M$ is free of rank $6$ over $B_M$ for all maximal
ideals $M$ of $B$, by Footnote 2 on p.\,1, we conclude that $G(A/B)$
is locally free of rank 6 over $B$, and Theorem~\ref{cubicase} follows.

\section{Behavior under products}\label{directsumsec}

%\begin{lemma}\label{nonzero}
%If $A$ is locally free of rank $n\geq 1$ over $B$, then $G(A/B)$ is not the zero ring.
%\end{lemma}

%Now s
Suppose $A_1,\ldots,A_k$ are locally free rings of rank $n_1,\ldots,n_k$, respectively, over $B$.  Then $A=A_1\times\cdots\times A_k$ is locally free of rank $n=n_1+\cdots +n_k$ over $B$.   To 
 prove Theorem~\ref{directsum}, we wish to understand
$G(A/B)$ in terms of $G(A_j/B)$ for $j\in\{1,\ldots,k\}$.

To this end, let $[k]$ denote the set $\{1,\ldots,k\}$.  For each $j\in [k]$, we define $e_j\in A$ by $e_1=(1,0,\ldots,0)$, $e_2=(0,1,0,\ldots,0)$, and so on.  For an $n$-tuple $i=(i_1,\ldots,i_n)\in[k]^n$, define $T_i\subset A^{\otimes n}$ to be the $B$-subalgebra generated by all pure tensors $a_1^{(1)}a_2^{(2)}\cdots a_n^{(n)}\in A^{\otimes n}$, where $a_m\in A_{i_m}e_{i_m}$ for all $m\in\{1,\ldots,n\}$ (so $T_i\cong
A_{i_1}\otimes\cdots\otimes A_{i_n}$ as $B$-algebras).  Then we have the decomposition
\begin{equation}\label{decompAn}
 A^{\otimes n} = \prod_{i\in [k]^n} T_i
\end{equation}
as $B$-algebras, where the $T_i$ factor in (\ref{decompAn}) corresponds to the idempotent
$e_i:=e_{i_1}^{(1)}e_{i_2}^{(2)}\!\cdots e_{i_n}^{(n)}\in A^{\otimes n}$.
% so that $T_i=e_iA^{\otimes n}$.
We also then have a corresponding decomposition
\begin{equation}\label{decompI}
I(A,B)= \prod_{i\in [k]^n} I_i(A,B),
%\bigoplus_{i\in [k]^n} I_i(A,B),
\end{equation}
where $I_i(A,B)=e_iI(A,B)$ is an ideal of the $B$-algebra $T_i$.  

Let $S$ denote the subset of all $n$-tuples $i=(i_1,\dots,i_n)\in[k]^n$ such that $n_1$ of the $i_m$'s are equal to 1, and $n_2$ of the $i_m$'s are equal to 2, etc.  (Thus, for any $i\in S$, we have an isomorphism $T_i\cong 
A_1^{\otimes n_1}\otimes \cdots \otimes A_k^{\otimes n_k}$ as $B$-algebras, obtained by appropriately permuting the tensor factors.)
For any $n$-tuple $i\notin S$, we claim that $I_i(A,B)=T_i$.  Indeed, for such an $i\notin S$, let $j\in[k]$ be an element that appears fewer than $n_j$ times as an entry in $i$.  Then since $e_j$ has characteristic polynomial $x^{n-n_j}(x-1)^{n_j}$, we have the relation 
\begin{equation}\label{iabrel}
1-\sum_{1\leq r_1< \cdots < r_{n_j}\leq n} e_j^{(r_1)}e_j^{(r_2)}\cdots e_j^{(r_{n_j})}   \in I(A/B).
\end{equation}
Multiplying (\ref{iabrel}) by any element of $T_i$, and noting that any pure tensor in $T_i$ has fewer than $n_j$ tensor factors in $A_je_j$, yields
\begin{equation}\label{iabrel2}
T_i \subseteq I_i(A,B)
\end{equation}
and so $T_i=I_i(A,B)$ as claimed.

%in $I(A,B)$, when multiplied wit
%define $R_i\subset G(A/B)$ as the image of $T_i$ in $G(A/B)$.  Then we claim that there is a natural decomposition
%\begin{equation}\label{decompGAB}
%G(A/B) = \bigoplus_{i\in S} R_i
%\end{equation}
%as $B$-algebras and, moreover, we have $R_i\cong G(A_1/B)\otimes\cdots\otimes G(A_k/B)$ for all $i\in S$.
%%, implying Theorem~\ref{directsum}.

%To see this, we first note that any general element of the form $a_1^{(1)}a_2^{(2)}\cdots a_n^{(n)}\in T_i$, such that $i\notin S$, is in fact in $I(A/B)$ and thus maps to zero in $G(A/B)$.   Indeed, let $j\in[k]$ be an element that appears fewer than $n_j$ times as an entry in $i$. 
%Then since $e_j$ has characteristic polynomial $x^{n-n_j}(x-1)^{n_j}$, we have the relation 
%$$1-\sum_{1\leq r_1< \cdots < r_{n_j}\leq n} e_j^{(r_1)}e_j^{(r_2)}\cdots e_j^{(r_{n_j})}   \in I(A/B).$$
%Multiplying by $a=a_1^{(1)}a_2^{(2)}\cdots a_n^{(n)}\in T_i$, we obtain 
%$$a\Bigl(1-\sum_{1\leq r_1< \cdots < r_{n_j}\leq n} e_j^{(r_1)}e_j^{(r_2)}\cdots e_j^{(r_{n_j})} \Bigr)\; =\; a \; \in\; I(A/B)$$
%%(since $a\in T_i$ implies that fewer than  $n_j$ of the $a_m$'s lie in $A_je_j$, so 
%since $a$ multiplied with any term in the above sum is zero.  We conclude that $a\in T_i$ maps to 0 in $G(A/B)$ if $i\notin S$, as claimed.

%Now if we write $A^{\otimes n}=\times_{i\in[k]^n} T_i$, then $I(A,B)=\times_{i\in[k]^n} I_i(A,B)$, where for each $i\in [k]^n$ we have that $I_i(A,B)$ is an ideal of the $B$-algebra $T_i$.  We have already seen that $I_i(A,B)=T_i$ if $i\notin S$.  
To determine $I_i(A,B)$ when $i\in S$, it suffices by symmetry (via the $S_n$-action on $A^{\otimes n}$) to consider the case $i=(i_1,i_2,\ldots,i_n)$, where $i_1=\cdots=i_{n_1}=1$, $i_{n_1+1}=\cdots=i_{n_1+n_2}=2$, 
 etc.
 %$\ldots\,$, $i_{n_1+\cdots+n_{k-1}+1}=\cdots=i_{n_1+\cdots+n_k}=k$. 
 To obtain generators for $I_i(A,B)$ for this particular $i\in S$,  
we take generators of $I(A,B)$ and project them onto $T_i$ by multiplying them by the idempotent 
$e_i.$  A natural generating set for $I(A,B)$ (by the work of Section~2)
is the set of all elements of the form 
\begin{equation}\label{fundrelss}
s_m(a)\;-\sum_{1\leq r_1< r_2< \cdots< r_m\leq n}
a^{(r_1)}a^{(r_2)}\cdots a^{(r_m)},\end{equation}
where $a=a_je_j$ for some $a_j\in A_j$ and $j\in[k]$.  If $a=a_je_j\in A_je_j$ is such an element, then multiplying the element (\ref{fundrelss}) in $I(A,B)$ by the idempotent $e_i=e_{i_1}^{(1)}e_{i_2}^{(2)}\!\cdots e_{i_n}^{(n)}\in T_i$, we obtain 0 unless $m\leq n_j$, in which case we obtain 
\begin{equation}\label{projfundrelss}
s_m(a_j) - \sum_{n_1+\cdots+n_{j-1}+1\leq r_1< r_2< \ldots< r_m\leq n_1+\cdots+n_j}
a^{(r_1)}a^{(r_2)}\cdots a^{(r_m)}.
\end{equation}
Thus the ideal in $T_i$ generated by all such elements is simply $I(A_1,B)\otimes\cdots\otimes I(A_k,B)\subset T_i = A_1^{\otimes n_1}\otimes\cdots\otimes A_k^{\otimes n_k}$.  It follows that $T_i/I_i(A,B)\cong G(A_1/B)\otimes\cdots\otimes G(A_k/B)$ for this particular $i$, and thus for any $i\in S$.

Therefore, if for each $i\in S$ we write $R_i:=T_i/I_i(A,B)$, then we have shown that
\begin{equation}\label{decompGAB}
G(A/B) = \prod_{i\in S} R_i
\end{equation}
as $B$-algebras and moreover, $R_i\cong G(A_1/B)\otimes\cdots\otimes G(A_k/B)$ for all $i\in S$, yielding Theorem~\ref{directsum}.

Note that the proof shows that the $S_n$-action on the $B$-module $G(A_1\times\cdots\times A_k/B)$ is induced from the $S_{n_1}\times\cdots\times S_{n_k}$-action on $G(A_1/B)\otimes \cdots\otimes G(A_k/B)$; i.e., we have
$$G(A_1\times\cdots\times A_k/B)\; \cong \; B[S_n]\otimes_{B[S_{n_1}\times\cdots\times S_{n_k}]}
\bigl[G(A_1/B)\otimes \cdots\otimes G(A_k/B)\bigr]$$ as $B[S_n]$-modules.

\section{An example of a ring of rank $4$ whose $S_4$-closure has rank
  $32>4!$}\label{degexample}
  
%Let $B$ be any base ring, and let $A$ be the ring $B[x,y,z]/(x,y,z)^2$
%of rank 4 over $B$.  We claim then that $G(A/B)$ has rank 32 over $B$.
In this section, we give an example of a ring $A$ of rank $4$ over
$B$---namely $A=B[x,y,z]/(x,y,z)^2$---such that $G(A/B)$ has rank
$32>4!$ over $B$.  
%This is contained in the following proposition;
%moreover, the proof 
%provides an explicit construction of $G(A/B)$.

\begin{proposition}
\label{prop:rnk32}
Let $B$ be a ring, and let $A$ be the 
ring $B[x,y,z]/(x,y,z)^2$ having rank $4$ over $B$.  
Then $G(A/B)$ is free of rank $32$ over $B$.
%The ring $A=B[x,y,z]/(x,y,z)^2$ of rank $4$ over $B$
\end{proposition}

\begin{proof}
Motivated by the relations (\ref{fundrels}) for $G(A/B)$, we give a
direct construction of
a ring $R$ over $B$, which we will then show to be naturally 
isomorphic to $G(A/B)$.
Precisely, we construct $R$ to have a $B$-module decomposition of the form
\begin{equation}\label{abcdecomp}
R= B\oplus [T(x) \oplus T(y) \oplus T(z)] \oplus
[U(x) \oplus U(y) \oplus U(z)] \oplus
[V(x,y) \oplus V(y,z) \oplus V(x,z)] \oplus W(x,y,z),
\end{equation}
where $T(\cdot)$, $U(\cdot)$, $V(\cdot,\cdot)$, and
$W(x,y,z)$ are free $B$-modules having ranks 3, 2, 5, and 1,
respectively.  Therefore, 
$R$ (and thus $G(A/B)$) 
will have $B$-rank $1+3\cdot 3+3\cdot 2 + 3\cdot 5 + 1 = 32$.
%We construct $R$ via the $B$-modules $T(\cdot)$, $U(\cdot)$,
%$V(\cdot,\cdot)$, and $W(\cdot,\cdot,\cdot)$.  
The constructions of these $B$-modules 
$T(\cdot)$, $U(\cdot)$,
$V(\cdot,\cdot)$, and $W(\cdot,\cdot,\cdot)$ 
are as follows.  

First, $T(x)$ is the $B$-module spanned by $x_1,x_2,x_3,x_4$, modulo
the relation $x_1+x_2+x_3+x_4=0$; $T(y)$ and $T(z)$ are defined
similarly, and hence each is three-dimensional.

Second, $U(x)$ is defined as the symmetric square of $T(x)$, modulo
the relations
$$x_1^2=x_2^2=x_3^2=x_4^2=x_1x_2+x_1x_3+x_1x_4+x_2x_3+x_2x_4+x_3x_4=0.$$
Now $x_1+x_2+x_3+x_4=0$ in $T(x)$, so multiplying 
by $x_1$ and $x_2$ respectively shows that
\[
x_1x_2+x_1x_3+x_1x_4=0 \quad\textrm{and}\quad x_1x_2+x_2x_3+x_2x_4=0
\]
respectively in $U(x)$; this in turn implies that 
\[
x_2x_3+x_3x_4+x_2x_4=0 \quad\textrm{and}\quad x_1x_3+x_3x_4+x_1x_4=0
\]
in $U(x)$.  Subtracting the first and last of the
latter four relations gives $x_1x_2=x_3x_4$, and similarly we have
$x_1x_3=x_2x_4$ and $x_1x_4=x_2x_3$.  We thus find that $U(x)$ is spanned
over $B$ by the images of any two of 
the three nonzero elements $x_1x_2$ (or $x_3x_4$),
$x_1x_3$ (or $x_2x_4$), and $x_1x_4$ (or $x_2x_3$), and any two of these are independent over $B$.  The $B$-modules $U(y)$ and
$U(z)$ are defined in the analogous manner, and are thus 
also two-dimensional.

Third, $V(x,y)$ is defined as the product $T(x) \otimes T(y)$, modulo the
relations 
\[
x_1 y_1=x_2 y_2=x_3 y_3=x_4 y_4=0
\]
(where we have suppressed the tensor symbols).  As $T(x)\otimes T(y)$ is a rank 9
module over $B$, we see that $V(x,y)$ is five-dimensional.  The
$B$-modules $V(y,z)$ and $V(x,z)$ are defined analogously, and hence are
also five-dimensional.

Finally, $W(x,y,z)$ is the space $T(x)\otimes T(y)\otimes T(z)$
modulo the relations 
\[
x_iy_iz_j=x_jy_iz_i=x_iy_jz_i=0
\]
for all $i$ and $j$, and the further relations
$$x_iy_jz_k=\sgn(i,j,k)x_1y_2z_3$$ 
for all permutations
$(i,j,k)$ of $(1,2,3)$,  We have imposed the latter relations 
in $W(x,y,z)$ because we have the relations
\[0=(x_4y_4)z_3=(-x_1-x_2-x_3)(-y_1-y_2-y_3)z_3=x_1y_2z_3+x_2y_1z_3,
\]
 in $I(A,B)$, implying $x_2y_1z_3=-x_1y_2z_3$, etc.  With these relations, we see
that the rank of $W(x,y,z)$ over $B$ is 1, and is spanned over $B$ by
$x_1y_2z_3$.

We have not defined any $B$-module components in $R$ involving
quadruple products of $x_i,y_j,z_k$ ($1\leq i,j,k\leq 3$) because
these we would like to be zero due to the relations
$x_iy_i=x_iz_i=y_iz_i=0$ in $A$.  Similarly, there are no $B$-module
components involving triple products of only $x_i$ and $y_j$ ($1\leq
i,j,k\leq 3$), since the analogues of 
such products in $G(A/B)$ would be zero:
\[0=x_1(y_1y_2+y_2y_3+y_1y_3)=x_1y_2y_3,\] and similarly all such
triple products would be zero in $G(A/B)$.  Thus we keep only those
components $T(\cdot)$, $U(\cdot)$, $V(\cdot,\cdot)$, and
$W(\cdot,\cdot,\cdot)$ appearing in (\ref{abcdecomp}).

%We must now check that $R$ is a ring.  This is easy, as 
%multiplication in 
The product structure on $R$ is defined simply in terms of 
the natural maps $T(x)\otimes
T(x)\to U(x)$, $T(x)\otimes T(y)\to V(x,y)$, $T(x)\otimes T(y)\otimes
T(z)\to W(x,y,z)$, $T(x)\otimes V(y,z)\to W(x,y,z)$, and so on.  All
other products (such as $T(x)\otimes V(x,y)$) are defined to be zero.
With this product structure, it is immediate that $R$ is a ring.  

To see that $G(A/B)\cong R$, we note that there is a natural surjective map
\[ A \otimes A\otimes A\otimes A\to R, \] sending $x^{(i)}\mapsto x_i$, $y^{(i)}\mapsto
y_i$, and $z^{(i)}\mapsto z_i$.  Furthermore, the kernel of this map is,
by design, contained in $I(A,B)$.  To see that it contains $I(A,B)$,
one may simply check that it contains the elements (\ref{fundrels}) on
a basis of $A$ over $B$, and so on the basis elements $1$ (trivial),
$x$, $y$, and $z$.  By symmetry of $x$, $y$, and $z$, we then only
need to check that the elements (\ref{fundrels}) are in the kernel for
$a=x$ and $j=1,2,3$, and this is again immediate.

We conclude that $G(A/B)\cong R$ has rank 32 over $B$.
\end{proof}

\section{Why do we need to allow the rank of $S_n$-closures to exceed
  $n!$\,\,?}\label{orderexample}

It is natural to ask if it is possible to enlarge the ideal $I(A,B)$ defined in (\ref{fundrels}) to an ideal $I'(A,B)$ such that: i) for any ring $A$ that is locally free of rank $n$ over a ring $B$,  the quotient $G'(A/B):=A^{\otimes n}/I'(A,B)$ is always locally free of rank $n!$ over $B$; and ii) the construction $G'(A/B)$ commutes with base change.  As the following theorem shows, the answer is no:

\begin{theorem}\label{newth}
Let $n\geq4$.  As $B$ ranges over all rings and 
%$A$ ranges over all $B$-algebras of rank~$n$, 
$A$ ranges over all rings that are locally free of rank~$n$ over $B$, 
there cannot exist corresponding rings $G'(A/B)$ that are locally free of rank~$n!$ over $B$ and $B$-algebra maps $i_1,\dots,i_n:A\to G'(A/B)$ such that:
\begin{itemize}
\item[{\rm (a)}]
for every $A$ and $B$, the images of $i_1,\dots,i_n$ generate $G'(A/B)$ as a $B$-algebra;
\item[{\rm (b)}]
$G'(A/B)$ and $i_1,\dots,i_n$ are functorial in $A$; 
\item[{\rm (c)}]
$G'(A/B)$ and $i_1,\dots,i_n$ respect base change, i.e., for all $B$-algebras~$C$, there is a functorial choice of isomorphism $G'(A/B)\otimes_B C
\xrightarrow{\sim}
%{\overset{\sim}{\rightarrow}} 
G'((A\otimes_B C)/C)$ which commute with the $i_j$;~and
\item[{\rm (d)}]
if $A/B$ is an $S_n$-extension of fields of degree $n$, then $G'(A/B)$ is isomorphic to the usual Galois closure of $A$ over $B$, and the $i_j$ correspond to the $n$ distinct embeddings of $A$.
\end{itemize}
\end{theorem}

\begin{remark}{\em 
Note that $G(A/B)$ satisfies all the essential properties (a)--(d) of the theorem (the map $i_j$ corresponds to $a\mapsto a^{(j)}$) but, as we have already seen, it does not always have rank $n!$ over $B$.  In the usual Galois closure $\tilde A$ of an $S_n$-extension $A/B$ of fields of
degree $n$, the maps $i_j$ correspond to the embeddings $A\hookrightarrow \tilde A$, whose images generate $\tilde A$ as an algebra over $B$.
}\end{remark}

\begin{prooft2}
%We begin by showing that any such construction of $G'(A/B)$ must have the property that $G'(A/B)$ is a quotient of $G(A/B)$.  Indeed, the maps $i_1,\dots,i_n:A\to G'(A/B)$ determine an $R$-algebra map $\pi':A^{\otimes n}\to G'(A/B)$.  By property $(a)$, the images of the $i_j$ generate $G'(A/B)$ as a ring, and so $\pi'$ is surjective.  Let $\pi:A^{\otimes n}\to G(A/B)$ denote the natural surjection.  We construct a morphism $\varphi:G(R/Z)\to G'(R/\Z)$ such that $\pi'=\varphi\circ\pi$.
%The purpose of this section is to prove Theorem \ref{newth}.  WHY DOES IT SUFFICE TO TREAT THE CASE $n=4$?  I'D LIKE TO SAY if $R/\Z$ is our rank 4 counter-example, then $R\oplus\Z^{n-4}$ is our rank $n$ counter-example??
%Let $p$ be any prime, 
We consider first the case $n=4$.  
Let $K$ be an $S_4$-quartic extension of $\Q$, and let $\O_K$ denote the ring of integers of $K$.  Fix a prime $p\geq 5$, and let $R$ be the ring $\Z+p\O_K$.
% be an order in $K$ such that 
Then $\bar{R}:=R\otimes_{\Z}\F_p\cong \F_p[x,y,z]/(x,y,z)^2$.  

We begin by showing that $G'(R/\Z)$ is a quotient of $G(R/\Z)$.  The maps $i_1,\dots,i_4:R\to G'(R/\Z)$ determine an $R$-algebra map $\pi':R^{\otimes 4}\to G'(R/\Z)$.  By property (a), the images of the $i_j$ generate $G(R/\Z)$ as a ring, and so $\pi'$ is surjective.  Now $R^{\otimes 4}\subset K^{\otimes 4}$, and since $G'(R/\Z)$ is free of rank~24 over~$\Z$, by property (c) it is a full rank $\Z$-submodule of $G'(K/\Q)$.  Furthermore, $G'(K/\Q)$ is isomorphic to $G(K/\Q)$ by Theorem~2 and property (d).   We thus see that 
$$I(R/\Z)\subset R^{\otimes4}\cap I(K/\Q) =R^{\otimes4}\cap\ker(K^{\otimes 4}\to G'(K/\Q))\subset \ker(\pi')$$
and so the map $\pi'$ factors through $R^{\otimes 4}/I(R,\Z) = G(R,\Z)$, as claimed.  

\begin{comment}
We begin with the case $n=4$.  
Let $\pi:R^{\otimes 4}\to G(R/\Z)$ denote the natural surjection.  
We construct a morphism $\varphi:G(R/\Z)\to G'(R/\Z)$ such that $\pi'=\varphi\circ\pi$.

By Theorem \ref{thm:functorial} and property $(c)$, we have $G(R/\Z)\otimes\Q\cong G(K/\Q)$ and $G'(R/\Z)\otimes\Q\cong G'(K/\Q)$.  By Theorem \ref{fieldcase} and property $(d)$, there exists an isomorphism $\varphi'$ which makes the diagram
\[
\xymatrix{
K^{\otimes 4}\ar_-{\pi\otimes\id}[d]\ar^-{\pi'\otimes\id}[dr] & \\
G(K/\Q)\ar^-{\varphi'}[r] & G'(K/\Q)
}
\]
commute.  Since $G'(R/\Z)$ is free, to define $\varphi$ it is enough to show that the image of $G(R/\Z)$ under $\varphi'$ is contained in the lattice $G'(R/\Z)$.  Since $\pi$ is surjective, this follows from the commutativity of the diagram
\[
\xymatrix{
R^{\otimes 4}\ar^-{\pi}[r]\ar[d] & G'(R/\Z)\ar[d]\\
K^{\otimes 4}\ar^-{\pi\otimes\id}[r] & G'(K/\Q)
}
\]
This shows the existence of $\varphi$, and hence that $G'(R/\Z)$ is a quotient of $G(R/\Z)$.
\end{comment}
By Theorem \ref{thm:functorial} and Proposition \ref{prop:rnk32}, the $\F_p$-algebra $G(R/\Z)\otimes\F_p$ is isomorphic to $G(\bar{R}/\F_p)$ and so has rank $32$ as an $\F_p$-module.  Now, by functoriality of the $S_4$-closure, the $\F_p$-module $G(\bar{R}/\F_p)$ is naturally a representation of the group $\Aut_{\F_p}(\bar{R})=\GL_3(\F_p)$, and also of $S_4$, over $\F_p$, and hence (since these actions commute) of the group $\Gamma=S_4\times \GL_3(\F_p)$.  Thus, by properties (b) and~(c), we see that $G'(\bar{R}/\F_p)\cong G'(R/\Z)\otimes \F_p$ is a $\Gamma$-equivariant quotient of $G(\bar{R}/\F_p)\cong G(R/\Z)\otimes\F_p$.

We use $\underline{\triv}$ and $\underline{\std}$ to denote the
trivial representation and the standard three-dimensional
representation of $\GL_3(\F_p)$, respectively.  Also, we write $\triv$,
$\sgn$, $\std$, $\std'$, and $\std_2$ to denote the trivial,
sign, standard, standard $\otimes$ sign, and 2-dimensional
$S_3$-standard representation of $S_4$, respectively.  These representations are
irreducible over $\F_p$ since $p\geq5$.  From the proof of Proposition~\ref{prop:rnk32}, 
we have the following decomposition of $G(\bar{R}/\F_p)$ into irreducible $\Gamma$-representations:
\begin{equation}\label{repdecomp}
G(\bar{R}/\F_p)\cong (\triv\otimes \underline{\triv})\oplus (\std\otimes\underline{\std})\oplus
(\std_2\otimes\underline{\std})\oplus(\std'\otimes \underline{\std}^\vee)\oplus 
(\std_2\otimes \underline{\std}^\vee)\otimes (\sgn\otimes\underline{\triv}).
\end{equation}
The $\F_p$-dimensions of these irreducible summands are 1, 9, 6, 9, 6, and 1 respectively,
giving a total of 32.  The first $\triv\otimes\underline{\triv}$ corresponds to
the subring $\F_p\subset \bar{R}$; we then observe that no sum of any
subset of elements of $\{9,6,9,6,1\}$ adds up to 8, and thus $G(\bar{R}/\F_p)$
has no $\Gamma$-equivariant quotient ring of rank 24 over $\F_p$. This proves the theorem in the case $n=4$.

A similar argument holds also when $n>4$.  Let 
%$R^\circ=R\oplus\Z^{n-4}$ and let $K^\circ=K\oplus\Q^{n-4}$.  By Lemma \ref{newl}, we have an isomorphism $\varphi:G(K^\circ/\Q)\to G'(K^\circ/\Q)$ which commutes with the surjections from $(K^\circ)^{\otimes n}$ induced by the $i_j$.  
$K$ be a degree $n$ field extension of $\Q$ with associated Galois group $S_n$, and let $\O_K$ denote the ring of integers of $K$.  Let $p> n$ be a prime such that $\O_K/p\O_K\cong \F_p^{n-4}\times Q$ for some quartic $\F_p$-algebra $Q$ (such a prime $p$ exists in any $S_n$-number field $K$ of degree $n$ by the Chebotarev density theorem).  Then $\F_p^{n-3}$ is a subring of $\O_K/p\O_K$.  Let $T$ be the preimage of the subring $\F_p^{n-3}\subset \O_K/p\O_K$ in $\O_K$, and let $R$ be the ring $T+p\O_K$.  
Then $\bar R:=R\otimes\F_p\cong \F_p^{n-4}\times\F_p[x,y,z]/(x,y,z)^2$.

By the identical argument as in the case $n=4$, we may deduce that $G'(R/\Z)$ is a quotient of $G(R/\Z)$.
%The maps $i_1,\dots,i_4:R\to G'(R/\Z)$ determine an $R$-algebra map $\pi':R^{\otimes 4}\to G'(R/\Z)$.  By property $(a)$, the images of the $i_j$ generate $G(R/\Z)$ as a ring, and so $\pi'$ is surjective.  Now $R^{\otimes 4}\subset K^{\otimes 4}$, and since $G'(R/\Z)$ is free of rank 24, by property $(c)$ it is a lattice in $G'(K/\Q)$, which in turn is isomorphic to $G(K/\Q)$ by Theorem~2 and property $(d)$.   We thus see that 
%$$I(R/\Z)\subset R^{\otimes4}\cap I(K/\Q) =R^{\otimes4}\cap\ker(K^{\otimes 4}\to G'(K/\Q))\subset \ker(\pi')$$
%and so the map $\pi'$ factors through $R^{\otimes 4}/I(R,\Z) = G(R,\Z)$, as claimed.  
%
%As in the case $n=4$, we may conclude $\varphi$ to a quotient map $\varphi':G(R^\circ/\Z)\to G'(R^\circ/\Z)$.
%
Furthermore, by Theorem~\ref{thm:functorial}, 
 the $\F_p$-algebra $G(R/\Z)\otimes\F_p$ is isomorphic to $G(\bar{R}/\F_p)$, and hence it has rank $32(n!/24)$ as an $\F_p$-module by Proposition \ref{prop:rnk32} and Theorem~\ref{directsum}.
%Now, by functoriality of the $S_4$-closure, the $\F_p$-module $G(\bar{R}/\F_p)$ is naturally a representation of the group $\Aut_{\F_p}(\bar{R})=\GL_3(\F_p)$, and also of $S_4$, over $\F_p$ and thus (since these actions commute) of the group $\Gamma=S_4\times \GL_3(\F_p)$.  By properties $(b)$ and~$(c)$, we see that $G'(\bar{R}/\F_p)\cong G'(R/\Z)\otimes \F_p$ is a $\Gamma$-equivariant quotient of $G(\bar{R}/\F_p)\cong G(R/\Z)\otimes\F_p$.
%
%By the proof of Theorem~\ref{directsum}, $G(R/\Z)\cong G(R/\Z)^{\oplus N}$ where $N=\frac{n!}{24}$, and t
By the proof of Theorem~\ref{directsum}, the action of $S_n$ on $G(\bar R/\F_p)$ is that induced from the action of $S_4$ on $G(\F_p[x,y,z]/(x,y,z)^2/\F_p)$.  Let $G'_1,\dots,G'_{n!/24}$ denote the images in $G'(\bar R/\F_p)$ of the $n!/24$ copies of $G(\F_p[x,y,z]/(x,y,z)^2/\F_p)$ in $G(\bar R/\F_p)$.
%, which naturally correspond to the elements of the coset space $S_n/S_4$. 
%as in Theorem~\ref{directsum}.  
%The $G'_i$ are submodules of $G'(R^\circ/\Z)$ and hence free.  Moreover, $\varphi$ is an isomorphism, so the $G'_i$ have the same rank.  
Then $G'_1$ yields a $\Gamma$-equivariant quotient of $G(\F_p[x,y,z]/(x,y,z)^2/\F_p)$ having rank 24.  This is again a contradiction.
\end{prooft2}

\section{The maximal rank of $S_n$-closures}
\label{sec:maxrank}

The purpose of this section is to show that the analogues for general
$n$ of the maximally degenerate ring of rank~4 (considered in
Section~\ref{degexample}) form the rings whose $S_n$-closures have
maximal rank.  Thus we prove Theorem~\ref{maxrank}.

The idea of our proof is as follows.  In a sense which we make precise
below, the ring $R_n=K[x_1,\dots,x_{n-1}]/(x_1,\dots,x_{n-1})^2$
is the ``maximally degenerate point'' in the moduli space of all rank~$n$ rings over $K$.  
%, so that every point in the moduli space specializes to it.  
Since Theorem \ref{thm:functorial} shows that the $S_n$-closures
of rank $n$ rings fit together into a nicely-behaved sheaf on the moduli
space, an upper semi-continuity argument allows us to conclude that
the rank of the $S_n$-closure is maximal at the degenerate ring $R_n$.

As in \cite{moduli}, let $\mathfrak{B}_n$ be the functor from
$\mathbf{Schemes}^{\textrm{op}}$ to $\mathbf{Sets}$ which assigns to
any scheme $S$ the set of isomorphism classes of pairs
$(\mathcal{A},\phi)$, where $\mathcal{A}$ is an $\OO_S$-algebra and
$\phi:\mathcal{A}\to\OO_S^n$ is an isomorphism of
$\OO_S$-modules.  By \cite[Prop.~1.1]{moduli}, the functor
$\mathfrak{B}_n$ is representable by an affine scheme of finite type
over $\Z$.

The base change $\mathfrak{B}_{n,K}$ of $\mathfrak{B}_n$ to $\Spec K$
is affine.  Write $\mathfrak{B}_{n,K}=\Spec B_n$.  The identity morphism
from $\mathfrak{B}_{n,K}$ to itself yields a distinguished isomorphism
class of pairs $(A_n,\phi)$ with $A_n$ a $B_n$-algebra and
$\phi:A_n\to B_n^n$ an isomorphism.  Let us choose an object
$(A_n,\phi)$ of this isomorphism class.  Since we are interested in
proving a statement about dimension, this choice does not matter.
Since the $S_n$-closure $G(A_n/K)$ of $A_n$ is a finitely-generated 
$B_n$-module, it defines a coherent sheaf $\mathcal{F}_n$ on
$\mathfrak{B}_{n,K}$.  By Theorem \ref{thm:functorial}, if we
have a morphism $f:\Spec C\to \mathfrak{B}_{n,K}$
corresponding to the pair $(R,\psi)$, then $f^*\mathcal{F}_n$ is
isomorphic to $G(R/C)$.

Note that there is a natural $\GL_{n,K}$-action on
$\mathfrak{B}_{n,K}$ and that Theorem \ref{thm:functorial} shows
that it extends to an action on the sheaf $\mathcal{F}_n$.
%By \cite[Cor 7.4]{moduli}, 
%every point of the moduli space $\mathfrak{B}_{n,K}$ specializes to 
%the $K$-point corresponding to $R_n$ with 
%the choice of ordered basis $(1,x_1,\dots,x_{n-1})$.  
The proof of \cite[Prop.~7.1]{moduli} shows that the $K$-point
corresponding to $R_n$ is in the Zariski closure of the
$\GL_{n,K}$-orbit of any other point.  Upper semi-continuity therefore
shows that the dimension of the fiber of $\mathcal{F}_n$ is maximal at
the point corresponding to $R_n$, as desired.
%; that is, the dimension of $G(R_n/K)$
%is maximal.
%Since the dimension of the fibers of $\mathcal{F}_n$ is upper 
%semi-continuous, we see then that the dimension of $G(R_n/K)$ is maximal.

\section{The $S_n$-closures of the degenerate rings $R_n$}
\label{deg}
\subsection{Preliminaries from $S_n$-representation theory}
In this subsection, we collect several facts from $S_n$-representation
theory that we use in the proof of Theorem~\ref{degtheorem} (made more
precise in Theorem~\ref{thm:main}).  

For us, given a 
positive integer $n$, a {\it partition of $n$} is an $n$-tuple
$\lambda=(\lambda_1,\dots,\lambda_n)$ satisfying 
$n\geq\lambda_1\geq\dots\geq\lambda_n\geq0$ and $\sum \lambda_i=n$.  We often drop 
the $\lambda_i=0$ in our notation, so that the partition $(3,1,0,0)$ of 4, 
for example, is denoted simply as $(3,1)$.  Partitions of $n$ play a key 
role in $S_n$-representation theory due to the following theorem (see, for
example, \cite[\S2.1.12]{repthy}).
\begin{theorem}
If $K$ is a field of characteristic $0$ or of characteristic $p>n$, then 
there is a canonical bijection between the set of 
partitions of $n$ and the set of isomorphism classes of irreducible $S_n$-representations over $K$.
\end{theorem}

Given a partition $\mu$, we denote by $V_\mu$ the
corresponding irreducible $S_n$-representation.  The $V_\mu$ are
called Specht modules and can, in fact, be defined over the integers.
%(see \cite[Def 2.3.4]{sagan}).  
We associate to $\mu$ a \emph{Young diagram}, which consists of
$n$ rows of boxes with $\mu_i$ boxes on the $i^{th}$ row.  For
example, the Young diagram of $\mu=(4,2,2,1)$ is
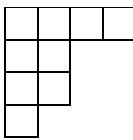
\[
\yng(4,2,2,1)
\]

A \emph{generalized Young tableau of shape} $\mu$ is a function $f$ 
which assigns a positive integer to every box of the Young diagram of $\mu$.  
We depict $f$ by drawing the Young diagram of $\mu$ and filling 
in each box with the positive integer assigned to it by $f$.  For example,
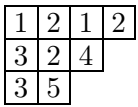
\[
\young(1212,324,35)
\]
is a generalized Young diagram of shape $(4,3,2)$.  
If $\lambda$ is another partition of $n$, then we say $f$ is a \emph{Young tableau of shape} 
$\mu$ \emph{and content} $\lambda$ if $f$ assigns the number $i$ to exactly $\lambda_i$ boxes.  
%If $\lambda$ and $\mu$ are two partitions of $n$, and if $k=|\{i:\lambda_i\neq0\}|$, then a \emph{Young tableau of shape} $\mu$ \emph{and content} $\lambda$ is an assignment to each box of the Young diagram of $\mu$ an element of $\{1,2,\dots,k\}$ in such a way that the element $i$ has been assigned to exactly $\lambda_i$ boxes.  
Such a Young tableau is called \emph{semi-standard} if the numbers assigned
to the boxes of the Young diagram of $\mu$ weakly increase across rows
and strongly increase down columns.  For example, both
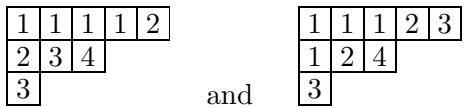
\[
\young(11112,234,3)\quad\textrm{\;and\;\;}\quad\young(11123,124,3)
\]
are Young tableaux of shape $(5,3,1)$ and content $(4,2,2,1)$, but
only the first is semi-standard.  
The \emph{Kostka\ number}
$K_{\mu\lambda}$ is defined to be the number of semi-standard Young
tableaux of shape $\mu$ and content $\lambda$.  
\begin{definition}
\emph{If $\lambda$ and $\mu$ are two partitions of $n$, we say $\mu$ \emph{dominates} $\lambda$ and write 
$\mu\triangleright\lambda$ if $\sum_{i=1}^{j}\mu_i\geq\sum_{i=1}^j\lambda_i$ for all $j$.}
\end{definition}

Note that in order for a Young tableau of shape $\mu$ and content
$\lambda$ to be semi-standard, the $\lambda_i$ boxes containing the 
number $i$ must be in the first $i$ rows.  So, if $\mu$ does not dominate $\lambda$,
then $K_{\mu\lambda}=0$.  The importance of the Kostka numbers is seen
in Young's Rule below (for a proof, see \cite[Cor.~4.39]{fh}).
\begin{theorem}$\emph{(Young's Rule)}$
\label{thm:young}
If $\lambda$ is a partition of $n$, then
\[
\emph{Ind}_{S_{\lambda_1}\times\dots\times S_{\lambda_n}}^{S_n}(\triv)=\bigoplus_{\mu\triangleright\lambda}K_{\mu\lambda}V_\mu.
\]
\end{theorem}

In particular, since 
\[
K[S_n]=\Ind_{S_1\times\dots\times S_1}^{S_n}(\triv)
\]
we see that $K_{\mu\lambda}=\dim V_\mu$, where $\lambda=(1,1,\dots,1)$.

There is a second combinatorial theorem we later make use of.  This
theorem, known as the hook formula, gives another way to relate $\dim
V_\lambda$ to the Young diagram of $\lambda$.
\begin{definition}
\emph{The \emph{hook number} of the $j^{th}$ box in the $i^{th}$ row of the Young diagram of $\lambda$ is 
$1+\lambda_i-j+|\{k\;:\;k>i,\lambda_k\geq\lambda_i\}|$.  That is, it is the number boxes in the ``hook'' which runs up 
the $j^{th}$ column, stops at the box in question, and continues across the $i^{th}$ row to the right.}
\end{definition}

For example, replacing each box in the Young diagram of $(4,2,2,1)$ by
its hook number, we have
\[
\young(7521,42,31,1)
\]
\begin{theorem}$\emph{(Hook Formula)}$ Given a partition $\lambda$ of
  $n$, let $H$ be the product of the hook numbers of the boxes in the
  Young diagram of $\lambda$.  Then $\dim V_\lambda=\frac{n!}{H}$.
\end{theorem}
For a proof, see \cite[Thm.~3.10.2]{sagan}.
%For example, we see that the dimensions of $V_{(n)}$ and $V_{(1,1,\dots,1)}$ are both $1$.  The Young Rule tells us that these are the trivial representation and the sign representation, respectively. (TAKE THIS OUT, OR GIVE AN EASY PROOF THAT (1,1,...,1) IS SIGN REPRESENTATION).

\subsection{A structure theorem for $S_n$-closures of degenerate
  rings}

Throughout this subsection, $K$ is a field of characteristic $0$ or of
characteristic $p>n$, and $R_n$ denotes the degenerate ring
$K[x_1,\dots,x_{n-1}]/(x_1,\dots,x_{n-1})^2$.  Then $R_n^{\otimes n}$
is a $K$-vector space of dimension $n^n$ with basis
$x_{i_1}\otimes\dots\otimes x_{i_n}$, where $i_j\in\{0,\dots,n-1\}$
and $x_0:=1$.  For notational convenience, we drop the tensor signs 
and let $I:=I(R_n,K)$.  Our goal in this subsection is to prove 
\begin{theorem}
\label{thm:main}
For each partition $\lambda$ of $n$, let $m_\lambda$ be the
multinomial coefficient $\binom{n-1}{k_0;\dots;k_{n-1}}$, where
$k_j=|\{i:i\neq1,\lambda_i=j\}|$.  Then there is an isomorphism
\[
G(R_n/K)\cong\bigoplus_{\substack{\mu\triangleright\lambda\\ \mu_1=\lambda_1}}m_\lambda K_{\mu\lambda}V_\mu
\]
of $S_n$-representations over $K$.
\end{theorem}

As we will show in Theorem \ref{cor:reg}, the theorem above implies that
the dimension of $G(R_n/K)$ over~$K$ is greater than $n!$ for $n\geq4$; that
is, it implies Theorem~\ref{degtheorem}.  As a first step in proving
Theorem~\ref{thm:main}, we begin by crudely decomposing $G(R_n/K)$
into certain naturally occurring $S_n$-representations parametrized
by partitions of $n$.

\begin{definition}
\emph{
Given an ordered partition $a=(a_0,a_1,\dots,a_{n-1})$ of $n$, let $M_a$ be the subrepresentation of $R_n^{\otimes n}$ spanned as a $K$-vector space by the elements
$x_{i_1}\cdots x_{i_n}$ with \,$|\{j:i_j=k\}|=a_k$.
}
\end{definition}

For example, writing $x$ and $y$ for $x_1$ and $x_2$, respectively, if $a=(1,2,1)$, then $M_a$ spanned over $K$ by the $12$ elements $1xxy$, $1xyx$, $1yxx$, $\dots$, $xx1y$, and $xxy1$.  Note that 
\[
%R_n^{\otimes n}=\bigoplus_{\stackrel{a\textrm{\ ordered}}{\textrm{partition\ of\ }n}}M_a
R_n^{\otimes n}=\bigoplus_a M_a
\]
as $S_n$-representations.

Let 
$\gamma(x_i\,,\,x_{i_1}\cdots x_{i_n})=\gamma_1(x_i\,,\,x_{i_1}\cdots x_{i_n})\in R_n^{\otimes n}$
where, for any $j\in\{1,\ldots,n\}$,  we set
\begin{equation}\gamma_j(x_i\,,\,x_{i_1}\cdots x_{i_n})\:=\:
%(x_i1\cdots1+1x_i\cdots1+\cdots+11\cdots x_i)^\ell
x_{i_1}\cdots x_{i_n}\cdot \sum_{1\leq i_1< \ldots< i_j\leq n}
x_i^{(i_1)}\cdots x_i^{(i_j)}.\end{equation}
Note that all such $\gamma_j(x_i\,,\,x_{i_1}\cdots x_{i_n})\in I$, because $s_j(x_i)=0$.

\begin{definition}
\emph{Given an ordered 
partition $a$ of $n$, let $I_a$ be the $K$-vector subspace of $I$ generated by all $\gamma(x_i\,,\,x_{i_1}\cdots x_{i_n})\in M_a$ with $i>0$.
}
\end{definition}
  
For example, again writing $x$ and $y$ for $x_1$ and $x_2$, if $a=(1,2,1)$ then $I_a$ is generated by the $6$ elements $\gamma(y,11xx)$, $\gamma(y,1x1x)$, $\dots$, $\gamma(y,xx11)$ as well as the $12$ elements $\gamma(x,11xy)$, $\gamma(x,1x1y)$, $\dots$, $\gamma(x,yx11)$.
%As shown in the proof of Theorem \ref{thm:functorial}, %As shown in Section ??, the ideal $I:=I(R_n,K)$ is generated by the relations (\ref{fundrels}) where $a$ ranges through a basis of $R_n$ over $K$.  The ideal $I$ is therefore generated as a $K$-vector space by the $\gamma(x_i,11\dots1)$.  As mentioned in the introduction, there is a natural $S_n$-action on $R_n^{\otimes{n}}$ given by permuting the tensor factors and this passes to an action on the $S_n$-closure $G(R_n/K)$ of $R_n$.  Here, we see that
%\[\pi(\gamma(x_k,x_{i_1}x_{i_2}\dots x_{i_n}))=\gamma(x_k,x_{i_{\pi^{-1}(1)}}x_{i_{\pi^{-1}(2)}}\dots x_{i_{\pi^{-1}(n)}})\]for all $\pi\in S_n$.
%
%Let $a=(a_0,a_1,\dots,a_{n-1})$ be an ordered partition of $n$ and let $M_a$ be the subrepresentation of $R_n^{\otimes n}$ generated by the $x_{i_1}\dots x_{i_n}$ with $a_k=|\{j:i_j=k\}|$.  Let $I_a$ be the subrepresentation of $I$ generated by the $\gamma(x_j,x_{j_1}\dots x_{j_n})\in M_a$.  For example, writing $x$ and $y$ for $x_1$ and $x_2$, respectively, if $a=(1,2,1)$, then $M_a$ is generated by the $12$ elements $1xxy$, $1xyx$, $1yxx$, $\dots$, $xx1y$, and $xxy1$; $I_a$ is generated by the $6$ elements $\gamma(y,11xx)$, $\gamma(y,1x1x)$, $\dots$, $\gamma(y,xx11)$ as well as the $12$ elements $\gamma(x,11xy)$, $\gamma(x,1x1y)$, $\dots$, $\gamma(x,yx11)$.

\begin{lemma}
If $a$ is an ordered partition of $n$, then $I\cap M_a=I_a$, and so
\[
G(R_n/K)=\bigoplus_a M_a/I_a.
\]
\end{lemma}
\begin{proof}
Clearly, $I_a$ is contained in $I\cap M_a$.
To prove the other containment, let 
$\beta\in I\cap M_a\subset I$.  By Section~2,
$I$ is generated as an ideal by the elements     
$\gamma_j(x_i, 11\cdots1)$ with $i,j > 0$, and therefore as a $K$-vector         
space by the $x_{i_1}\cdots x_{i_n}\cdot\gamma_j(x_i, 11\cdots1)
=\gamma_j(x_i\,,\,x_{i_1}\cdots x_{i_n})$ with $i,j >  0$.  
%It follows that we may write
%\[
%\beta=\sum_{k=1}^N\lambda_i\alpha_i,
%\]
%where $\lambda_i\in K$ and $\alpha_i$ is some $\gamma_j(x_i\,,\,x_{i_1}\cdots x_{i_n})\in M_{a(i)}$ with $i,j>0$, and $a(i)$ is some ordered 
%partition of $n$.  Since $R_n^{\otimes n}$ is a direct sum of the $M_{a'}$, we see that for $a'\neq a$, 
Since each $\gamma_j(x_i\,,\,x_{i_1}\cdots x_{i_n})$ is contained in $M_{a'}$ for some ordered partition $a'$ of $n$, we see that 
%\[
%\sum_{a(i)=a'}\lambda_i\alpha_i=0.
%\]
%Therefore, 
$\beta$ can be expressed as a $K$-linear combination of the 
%$\alpha_i\in M_a$, and hence of the 
various $\gamma_j(x_i\,,\,x_{i_1}\cdots x_{i_n})\in M_a$.

To prove the lemma, it remains to show that the $\gamma_j(x_i\,,\,x_{i_1}\cdots x_{i_n})\in I_a$ for $j\in\{1,\ldots,n\}$ can be expressed as $K$-linear combinations of the $\gamma(x_i\,,\,x_{i_1}\cdots x_{i_n})\in I_a$.  To see this, it suffices to note that $\gamma_j(x_i\,,\,x_{i_1}\cdots x_{i_n})\in I_a$, for any $j\in\{1,\ldots,n\}$, can be expressed as
\[ \gamma_j(x_i\,,\,x_{i_1}\cdots x_{i_n})=\frac1j\cdot\gamma_{j-1}(x_i\,,\,x_{i_1}\cdots x_{i_n})\cdot\sum_{k=1}^n x_i^{(k)}.\]
It follows by induction on $j$ that any $\gamma_j(x_i\,,\,x_{i_1}\cdots x_{i_n})\in M_a$ ($j\in\{1,\ldots,n\}$) can be expressed as a $K$-linear combination of the $\gamma(x_i\,,\,x_{i_1}\cdots x_{i_n})\in I_a$, proving the lemma.
%This proves the claim.
%, and as a result,\[G(R_n/K)=\bigoplus_a M_a/I_a.\]
  \end{proof}

Our next lemma shows that if $a_0<a_k$ for some $k$, then $M_a=I_a$.
\begin{lemma}
\label{l:incexc}
Let $i_1,\dots,i_n\in\{0,1,\dots,n-1\}$.  If there is some $k$ such that 
\[
|\{j:i_j=0\}|<|\{j:i_j=k\}|,
\]
then $x_{i_1}\cdots x_{i_n}\in I$.
\end{lemma}

Since the notation in the proof of this lemma is a bit cumbersome, we
first illustrate the proof with a specific example.  Denoting $x_1$,
$x_2$, and $x_3$ by $x$, $y$, and $z$, respectively, let us show
$1yx1xzyx\in I$.  For $a,b,c\in S:=\{1,3,4,5,8\}$, let $[a,b,c]$
denote $x_{i_1}\cdots x_{i_8}$ with $i_a=i_b=i_c=1$,
$i_2=i_7=2$, $i_6=3$, and all other $i_j=0$.  For example,
$[1,3,4]=xyxx1zy1$ and $[3,5,8]=1yx1xzyx$.  By the inclusion-exclusion
principle, we have
\[
\begin{array}{rcrcrcrcrcr}
1yx1xzyx&=&\displaystyle{\sum_{\substack{a<b<c\\ a,b,c\in S}}[a,b,c]} &-& 
\displaystyle\sum_{\substack{a<b\\ a,b\in S-\{1\}}}[1,a,b] &-& 
\displaystyle\sum_{\substack{a<b\\ a,b\in S-\{4\}}}[4,a,b] &+& \displaystyle\sum_{a\in S-\{1,4\}}\;[1,4,a] \\[.36in]
 &=& \gamma_3(x,1y111zy1) &-&\gamma_2(x,xy111zy1) &-&\gamma_2(x,1y1x1zy1) &+&\gamma_1(x,xy1x1zy1)\,.
 \end{array}
\]
Hence $1yx1xzyx\in I$.  
%For example, $\sum[1,a,b]=\gamma_2(x,xy111zy1)\in I$.
% where 
%$\alpha$ is the sum of all elements of the form $x_{i_1}\cdots x_{i_8}$ with exactly two of the $i_j=1$ and all other 
%$i_j=0$.  Hence $\alpha=

\vspace{.125in}
%We now give the proof of Lemma \ref{l:incexc}.
\begin{proofl}
%[Proof of Lemma \ref{l:incexc}]
Let $S=S_0\cup S_k$, where $S_0=\{j:i_j=0\}$ and $S_k=\{j:i_j=k\}$.  
%Let $m=|S_k|$, so that $|S_0|<m$.  
%For 
%%distinct elements 
%%$a_1,\dots,a_{m}\in S$, let $[a_1,\dots,a_{m}]$ 
%a subset $T\subset S$ such that $|T|=|S_k|$, let $x_T$ denote the element $x_{i'_1}\cdots x_{i'_n}$, where 
%\begin{equation*}
%i'_j = \left\{
%\begin{array}{ll}
%0 & \mbox{if $j\in S-T$;}\\
%k & \mbox{if $j\in T$; and}\\
%i_j & \mbox{otherwise.}
%\end{array}
%\right.
%\end{equation*}
%with $i'_j=i_j$ if $j\notin S$, and with all other $i'_j=0$.  By an inclusion-exclusion argument similar to the one 
%above, we are reduced to showing
Then, by the inclusion-exclusion principle, we have
\begin{equation}
\begin{array}{rcl}
x_{i_1}\cdots x_{i_n}\;\;=\;\;
\;\displaystyle\prod_{j}x_{i_j}^{(j)}&=&
\displaystyle\sum_{U\subset S_0}(-1)^{|U|}\!
\displaystyle\sum_{\substack{U\subset T\subset S\\|T|=|S_k|}}
\;\prod_{j\in T}x_k^{(j)}\displaystyle\prod_{j\notin S} x_{i_j}^{(j)}\\[.36in]
&=&\displaystyle\sum_{U\subset S_0}(-1)^{|U|} \gamma_{|S_k|-|U|}\Bigl(x_k,\prod_{j\in U}x_k^{(j)}
\displaystyle\prod_{j\notin S} x_{i_j}^{(j)}\Bigr)\,.
\end{array}
\end{equation}
Hence $x_{i_1}\cdots x_{i_n}\in I$.
%\begin{equation}
%x_{i_1}\cdots x_{i_n}\;\;=\;\;\sum_{U\subset S_0}(-1)^{|U|}\sum_{\substack{U\subset T\subset S\\|T|=|S_k|}}
%x_T\,.
%\end{equation}
%\[
%\sum_{\substack{a_1<\dots<a_c\\ a_j\in S- \{b_j\}}}[a_1,\dots, a_c,b_1,\dots, b_{|T|-c}]\in I,
%\]
%We claim that the inner sum is in $I$ for each subset $U\subset S_0$, thus implying that $x_{i_1}\cdots x_{i_n}\in I$.  To see that 
%where $1\leq c< |T|$ and the $b_j$ are fixed elements of $S$.  This sum equals 
%$\alpha\cdot x_{i'_1}\cdots x_{i'_n}$, where $i'_j=i_j$ if $j\notin S-\{b_\ell\}$, and $i_j=0$ otherwise; 
%here $\alpha$ is the sum of all $x_{i''_1}\cdots x_{i''_n}$ with exactly $|T|-c$ of the $i''_j=k$.  Since $|T|-c>0$, 
%we see $\alpha\in I$, and hence the sum is as well.
\end{proofl}

We see from Lemma \ref{l:incexc} that
\[
G(R_n/K)=\bigoplus_{\substack{a\textrm{\ s.t.}\\ a_0\geq a_k\textrm{\ }\forall k}} M_a/I_a.
\]
If $\sigma$ is a permutation of $\{0,1,\dots,n-1\}$, and $a=(a_0,\dots,a_{n-1})$ is an ordered partition of $n$, 
then let $\sigma(a):=(a_{\sigma^{-1}(0)},\dots,a_{\sigma^{-1}(n-1)})$.  Note that if $\sigma$ fixes $0$, then it 
defines an isomorphism of $S_n$-representations $M_a\to M_{\sigma(a)}$ 
by sending $x_{i_1}\cdots x_{i_n}$ to $x_{\sigma(i_1)}\cdots x_{\sigma(i_n)}$.  We remark that if $\sigma$ 
does not fix $0$, it still defines an isomorphism of vector spaces, but this is in general not an isomorphism of 
$S_n$-representations.  
%For example, if $a_0\geq a_k$ for all $k$ and $a_0> a_{\sigma^{-1}(0)}$, then 
%Lemma \ref{l:incexc} shows that $M_{\sigma(a)}=I_{\sigma(a)}$; however, it follows from $\eqref{eq:star}$ and
%Proposition \ref{prop:proper} below that $I_a$ is a $\emph{proper}$ subrepresentation of $M_a$.

Now let $a=(a_0,\dots,a_{n-1})$ be an ordered partition of $n$ such that 
$a_0\geq a_k$ for all $k$.  For all $j\in\{0,\ldots, n-1\}$, let 
$k_j=|\{i:i\neq0,\;a_i=j\}|$.  Then $\{\sigma(a):\sigma(0)=0\}$ has cardinality 
$\binom{n-1}{k_0;\dots;k_{n-1}}=m_{\lambda(a)}$, where $\lambda(a):=(\lambda_1,\dots,\lambda_n)$ is the (unordered) partition of 
$n$ such that $\{\lambda_i:1\leq i\leq n\}=\{a_i:0\leq i\leq n-1\}$ as multisets.

\begin{definition}
\emph{
For any partition $\lambda$ of $n$, let $M_\lambda$ and $I_\lambda$ denote the isomorphism classes of the 
$S_n$-representations $M_a$ and $I_a$, respectively, where $a=(a_0,\dots,a_{n-1})$ is any ordered partition of $n$ such that $a_0=\lambda_1$ and $\{\lambda_i\}=\{a_i\}$ 
as multisets.  This is well-defined as $\lambda(a)=\lambda(\sigma(a))$ for all $\sigma$ fixing $0$.
}
\end{definition}

Since each partition $\lambda$ corresponds to $m_\lambda$ ordered partitions $a$ in the above definition, we obtain
%Since $a\mapsto\lambda(a)$ gives a bijection of $\{a:a_0\geq a_k\textrm{\ }\forall k\}$ with the set of partitions of $n$, 
%we have 
%shown

\begin{proposition}
For all partitions $\lambda$ of $n$, let $m_\lambda$ be the
multinomial coefficient $\binom{n-1}{k_0;\dots;k_{n-1}}$, where
$k_j=|\{i:i\neq1,\lambda_i=j\}|$.  Then there is an isomorphism
\[
G(R_n/K)\cong\bigoplus_\lambda m_\lambda M_\lambda/I_\lambda
\]
of $S_n$-representations.
\end{proposition}

Given a partition $\lambda$ of $n$, let $i_j=k$ if $\sum_{m=1}^{k-1} \lambda_m < j \leq \sum_{m=1}^k \lambda_m$.  
%so that $i_1=\dots=i_{\lambda_1}=0$, $i_{\lambda_1+1}=\dots=i_{\lambda_1+\lambda_2}=1$, and so on.  (GET RID OF "so that... and so on"??)  
Then note that
\[
M_\lambda=\Ind_{S_{\lambda_1}\times\cdots\times S_{\lambda_n}}^{S_n}(K\cdot x_{i_1}\cdots x_{i_n}).
\]
Since $K\cdot x_{i_1}\cdots x_{i_n}$ is the trivial representation of $S_{\lambda_1}\times\cdots\times S_{\lambda_n}$, by 
Young's Rule we have
\begin{equation}
\label{eq:star}
M_\lambda\cong\bigoplus_{\mu\tr\lambda}K_{\mu\lambda}V_\mu,
\end{equation}
where $\lambda$ runs through the partitions of $n$.  We have therefore reduced Theorem $\ref{thm:main}$ to the following:
%theorem.

\begin{theorem}
\label{thm:ideal}
For all partitions $\lambda$ of $n$, we have
\[
I_\lambda\cong\bigoplus_{\substack{\mu\tr\lambda\\ \mu_1>\lambda_1}} K_{\mu\lambda}V_\mu
\]
as $S_n$-representations.
\end{theorem}

To prove Theorem \ref{thm:ideal}, we show that $I_\lambda$ contains 
%a copy of 
%$\bigoplus_{\mu\tr\lambda,\;\mu_1>\lambda_1} K_{\mu\lambda}V_\mu$ 
$K_{\mu\lambda}$ copies of $V_\mu$ if $\mu\tr\lambda$ and $\mu_1>\lambda_1$,
and that it contains no copy of 
$V_\mu$ if $\mu\tr\lambda$ but $\mu_1=\lambda_1$.  These two statements are the content of 
Propositions~$\ref{prop:first}$ and $\ref{prop:proper}$, respectively.
\begin{proposition}
\label{prop:first}
If $\lambda$ and $\mu$ are partitions of $n$ with $\mu_1>\lambda_1$, then the natural morphism of $S_n$-representations
\[
\Hom(V_\mu,I_\lambda)\longrightarrow\Hom(V_\mu,M_\lambda)
\]
is an isomorphism.
\end{proposition}

\begin{proof}
  Given a semi-standard Young tableau $T$ of shape $\mu$ and content
  $\lambda$, if $i=j+\sum_{m=1}^{k-1}\lambda_m <
  \sum_{m=1}^{k}\lambda_m$ for $j>0$, then let $T(i)$ be the number
  assigned to the $j^{th}$ box on the $k^{th}$ row of $T$.  For
  example, $T(\lambda_1+1)$ is the number assigned to the first box of
  the second row.  We can associate to $T$ an element
  $\alpha(T):=x_{T(1)-1}\cdots x_{T(n)-1}$ of $M_\lambda$.  Let $A_T$
  be the set of Young tableaux $T'$ of shape $\mu$ and content
  $\lambda$ such that for all $i$, the multiset of numbers in the
  $i^{th}$ row of $T'$ is the same as the multiset of numbers in the
  $i^{th}$ row of $T$.  Then by \cite[2.10.1]{sagan}, the image of any
  morphism $V_\mu\to M_\lambda$ of $S_n$-representations is
  contained in the $S_n$-subspace of $M_\lambda$ generated by the
  elements $\sum_{T'\in A_T} \alpha(T')$ as $T$ ranges over the
  semi-standard Young tableaux of shape $\mu$ and content $\lambda$.

  It therefore suffices to show $\sum_{T'\in A_T} \alpha(T')\in
  I_\lambda$ for every semi-standard Young tableau $T$ of shape $\mu$
  and content $\lambda$.  We define an equivalence relation on $A_T$
  by $T'\sim T''$ if $T'(i)=T''(i)$ for all $i>\mu_1$.  This
  equivalence relation partitions $A_T$ into a disjoint union of
  sets $S_1, S_2,\dots,S_\ell$.  For $i>\mu_1$, let $S_j(i)=T'(i)$ for
  any $T'\in S_j$.  Since
\[
\sum_{T'\in A_T}\alpha(T')=\sum_{j=1}^\ell\sum_{T'\in S_j}\alpha(T'),
\]
it suffices to show each $\sum_{T'\in S_j}\alpha(T')\in I_\lambda$.  Note that
\[
\sum_{T'\in S_j}\alpha(T')=\delta\cdot (\underbrace{11\cdots1}_{\mu_1}x_{S_j(\mu_1+1)-1}\cdots x_{S_j(n)-1}),
\]
where $\delta$ is the sum of all elements of the form $x_{i_1}\cdots
x_{i_n}$ with
\[
\{i_k\}=\{T(k)-1:1\leq k\leq \mu_1\}\cup\{\underbrace{0,0,\dots,0}_{n-\mu_1}\}
\]
as multisets.  Letting $a_m=|\{k:i_k=m\}|$ and noting that there is
some $m\neq0$ for which $a_m>0$, we see that
\[
\delta=\prod_{m=1}^{n-1}\gamma_{a_m}(x_m,11\cdots1)\in I_\lambda,
%.\qedhere
\]
which finishes the proof.
\end{proof}
\begin{proposition}
\label{prop:proper}
If $\mu\tr\lambda$ and $\mu_1=\lambda_1$, then $V_\mu$ does not occur
in $I_\lambda$.
\end{proposition}
\begin{proof}
  Let $\Gamma_m$ be the subrepresentation of $I_\lambda$ generated by
  the $\gamma(x_m,x_{i_1}\cdots x_{i_n})\in I_\lambda$.  Let $\ell$ be
  the smallest integer greater than or equal to $m$ such that
  $\lambda_m=\lambda_\ell>\lambda_{\ell+1}$.  If $\lambda_m=\lambda_j$
  for all $j\geq m$, then let $\ell=n$.  We define
\[
\lambda'=(\lambda'_1,\ldots,\lambda'_n)=(\lambda_1+1,\lambda_2,\dots,\lambda_{\ell-1}, \lambda_\ell-1,\lambda_{\ell+1},\dots,\lambda_n).
\]
Let $i_j = m$ if $\sum_{b=1}^{m} \lambda'_b < j \leq \sum_{b=1}^{m+1} \lambda'_b$.  Then note that
\[
\Gamma_m=\Ind_{S_{\lambda'_1}\times\cdots\times S_{\lambda'_n}}^{S_n} (K\cdot\gamma(x_m,x_{i_1}\cdots x_{i_n})).
\]
Since $K\cdot\gamma(x_m,x_{i_1}\cdots x_{i_n})$ is the trivial representation of $S_{\lambda'_1}\times\cdots\times S_{\lambda'_n}$, 
Young's Rule tells us that
\[
\Gamma_m=\bigoplus_{\epsilon\tr\lambda'}K_{\epsilon\lambda'}V_\epsilon.
\]
Since $\lambda'_1=\lambda_1+1$, any $\epsilon$ which dominates
$\lambda'$ must have $\epsilon_1>\lambda_1$.  Therefore $V_\mu$ does
not occur in any of the $\Gamma_m$, and since $I_\lambda$ is the
$K$-vector space span of the $\Gamma_m$, we conclude that $V_\mu$ does not occur in $I_\lambda$.
\end{proof}

This concludes the proof of Theorem $\ref{thm:ideal}$, and hence also
of Theorem \ref{thm:main}.  We now turn to the following theorem, which implies 
Theorem $\ref{degtheorem}$.

\begin{theorem}
\label{cor:reg}
The regular representation is a subrepresentation of $G(R_n/K)$.  If $n\geq4$, it is a proper subrepresentation.  %In particular, 
%Theorem $\ref{degtheorem}$ follows.
%As a result, $\dim_K G(R_n/K)\geq n!$ and this inequality is strict if $n\geq4$.
\end{theorem}

As the proof of this theorem shows, as $n$ gets large, the regular
representation is only a small subrepresentation, and so the bound in
Theorem \ref{degtheorem} is a weak one.
%Nonetheless the corollary accomplishes our goal of showing that $S_n$-closures 
%of ring extensions can have ranks which are larger than one might initially expect (???).
\begin{lemma}
\label{l:hook}
Let $\epsilon$ and $\tau$ be two partitions of $n$.  Suppose
$\tau_{k-1}>\tau_k=0$ and that
$\epsilon=(\tau_1,\dots,\tau_{i-1},\tau_i-1,\tau_{i+1},\dots,\tau_{k-1},1)$
for some $i>1$.  Let $E_1$ and $T_1$ be the product of the hook
numbers of the boxes in the first row of the Young diagram of
$\epsilon$ and $\tau$, respectively.  Then $T_1\geq E_1$.
\end{lemma}
\begin{proof}
  Let $h_1$ and $h_2$ be the hook numbers of the first and
  $\tau_i^{\phantom{1} th}$ box in the first row of the Young diagram
  of $\tau$, respectively.  Then
\[
E_1=T_1\frac{(h_1+1)(h_2-1)}{h_1h_2}.
\]
Expanding $(h_1+1)(h_2-1)$, and noting that $h_1>h_2$, yields
the desired inequality.
\end{proof}

\begin{prooft}
%[Proof of Theorem \ref{cor:reg}]
We must show that
\[
\sum_{\substack{\mu\tr\lambda\\ \mu_1=\lambda_1}}m_\lambda K_{\mu\lambda}\geq\dim V_\mu.
\]
Fix $\mu$ and let $\lambda=(\mu_1,1,\dots,1)$.  We in fact prove that $m_\lambda K_{\mu\lambda}\geq\dim V_\mu$.

If $\mu=(n)$, then $\lambda=\mu$ and $m_\lambda K_{\mu\lambda}=1=\dim
V_\mu$.  Now suppose $\mu_1<n$.  Let $\mu'=(\mu_2,\dots,\mu_n)$ and
$\lambda'=(\lambda_2,\dots,\lambda_n)$.  Since $\mu_1=\lambda_1$, the
first row of every semi-standard Young tableau of shape~$\mu$ and
content $\lambda$ consists entirely of 
%$\mu_1$ 
1's.  Therefore,
\[
K_{\mu\lambda}=K_{\mu'\lambda'}=\dim V_{\mu'},
\]
where the second equality comes from the paragraph following Theorem
\ref{thm:young}.  Let $H$ be the product of the hook numbers of the
Young diagram of $\mu$ and let $H_1$ be the product of the hook
numbers of the boxes in the first row.  Since
\[
\dim V_\mu=\frac{n!}{H}=\dim V_{\mu'}\cdot\frac{n!}{H_1(n-\mu_1)!}
\]
and $m_\lambda=\binom{n-1}{\mu_1-1}$, we need only show $H_1\geq
n(\mu_1-1)!$.  Note that the product of the hook numbers of the boxes
in the first row of the Young diagram of $\lambda$ is $n(\mu_1-1)!$.
The first part of the theorem therefore follows by successively applying Lemma~\ref{l:hook}.

%To finish the proof of the corollary, note 
Note that if $n\geq4$, then letting $\mu=(n-2,2)$ and
$\lambda=(n-2,1,1)$, we obtain
\[
\sum_{\substack{\mu\tr\lambda\\ \mu_1=\lambda_1}}m_\lambda K_{\mu\lambda}=m_\lambda K_{\mu\lambda}+m_\mu K_{\mu\mu}>
\dim V_\mu,
%\qedhere
\]
which shows that the regular representation is a proper
subrepresentation.
\end{prooft}

\subsection{Examples}

\label{sec:ex}
%The cases $n=3$ and $n=4$ of the structure theorem are illustrated in the table below.  Recall (in down to earth terms) 
%what $m_\lambda$ is: you remove the first row of $\lambda$ and then count how many remaining rows have 1 box, how 
%many have 2 boxes, etc. We also want to count how many have 0 boxes.  We then take $(n-1)!$ divided by the product of 
%these numbers, e.g. if $\lambda$ is 
%\[\yng(2,2),\]then $m_\lambda=\frac{3!}{2!0!2!}=3$ (remember that $\lambda$ is really $(2,2,0,0)$).
In this section we illustrate Theorem \ref{thm:main} in the cases $n=3$ and $n=4$.  
%The cases $n=3$ and $n=4$ of Theorem \ref{thm:main} are illustrated in tables below.
The following table collects the relevant information when $n=3$.
\[
\begin{array}{|l | c | l | c | c | c|}\hline
\mu & \dim V_\mu & \lambda\textrm{\ s.t.\ } \lambda_1=\mu_1\textrm{\
  and\ } \mu\tr\lambda & m_\lambda & K_{\mu\lambda} 
& m_\lambda K_{\mu\lambda}\\
\hline
 & & & & & \\[-.085in]
\yng(3) & 1 & \hspace{.51in}\yng(3) & 1 & 1 & 1\\[.05in]\hline
 & & & & & \\[-.085in]
 \yng(2,1) & 2 & \hspace{.51in}\yng(2,1) & 2 & 1 & 2\\[.05in] \hline
 & & & & & \\[-.085in]
\yng(1,1,1) & 1 & \hspace{.51in}\yng(1,1,1) & 1 & 1 & 1\\[.05in] \hline
\end{array}
\]
We see that for each partition $\mu$ of $3$, the dimension of $V_\mu$ agrees with $m_\mu K_{\mu\mu}$ and so 
Theorem \ref{thm:main} shows that $G(R_3/K)$ is the regular representation.

The cases $n\leq 3$ are rather uninteresting since for such $n$, whenever $\mu$ and $\lambda$ are partitions of $n$ with 
$\mu$ dominating $\lambda$ and $\lambda_1=\mu_1$, we in fact have $\mu=\lambda$.  When $n=4$, however, there exists a 
single pair $(\mu,\lambda)$ of partitions satifying the above conditions for which $\mu$ and $\lambda$ are distinct.  
As shown in Theorem \ref{cor:reg}, this forces $G(R_4/K)$ to properly contain a copy of the regular 
representation.  The $n=4$ case is summarized in the table below.
\[
\begin{array}{|l | c | l | c | c | c|}\hline
\mu & \dim V_\mu & \lambda\textrm{\ s.t.\ } \lambda_1=\mu_1\textrm{\ and\ } \mu\tr\lambda & m_\lambda & K_{\mu\lambda} 
& m_\lambda K_{\mu\lambda}\\
\hline
 & & & & & \\[-.085in] 
\yng(4) & 1 & \hspace{.5in}\yng(4) & 1 & 1 & 1\\[.05in]\hline
 & & & & & \\[-.085in] 
 \yng(3,1) & 3 & \hspace{.5in}\yng(3,1) & 3 & 1 & 3\\[.05in]\hline
 & & & & & \\[-.085in] 
 \yng(2,2) & 2 & \hspace{.5in}\yng(2,2) & 3 & 1 & 3\\[.03in]
  & & & & & \\[-.085in] 
  & & \hspace{.5in}\yng(2,1,1) & 3 & 1 &3 \\[.05in]\hline
 & & & & & \\[-.085in] 
 \yng(2,1,1) & 3 & \hspace{.5in}\yng(2,1,1) & 3 & 1 & 3\\[.05in]\hline
 & & & & & \\[-.085in] 
 \yng(1,1,1,1) & 1 & \hspace{.5in}\yng(1,1,1,1) & 1 & 1 & 1\\[.05in] \hline
\end{array}
\]
We see then from Theorem \ref{thm:main} that $G(R_4/K)$ contains
exactly $\dim V_\mu$ copies of $V_\mu$ for every partition $\mu$ of
$4$ other than $\mu=(2,2)$.  We see, however, that $G(R_4/K)$ contains
$6$ copies of $V_{(2,2)}$.  It follows that $G(R_4/K)$ is the regular
representation direct sum 4 copies of $V_{(2,2)}$.  Since $V_{(2,2)}$
is 2-dimensional, we see $G(R_4/K)$ has dimension $24+8=32$.

Let us now reconcile the decomposition of $G(R_4/K)$ given by Theorem
\ref{thm:main} with the explicit decomposition given in Section
\ref{degexample}.  We make no assumption here on the characteristic of
$K$.  Recall that $T(x)$ has generators $x_i$ for $1\leq i\leq 4$
satisfying the relation $\sum x_i=0$ and that $\sigma\in S_4$ acts by
$\sigma(x_i)=x_{\sigma(i)}$.  We see then that $T(x)$ is the standard
representation; that is, $T(x)\cong V_{(3,1)}$.  Recall that $U(x)$
is a two-dimesional vector space generated by the equivalence classes
of
\[
x_1y_2+x_2y_1+x_3y_4+x_4y_3 \quad\textrm{and}\quad x_1y_3+x_3y_1+x_2y_4+x_4y_2
\]
with $S_4$-action given by $\sigma(x_i)=x_{\sigma(i)}$ and
$\sigma(y_i)=y_{\sigma(i)}$.  Letting $H$ be the subgroup of $S_4$
generated by $(12)(34)$ and $(13)(24)$, we see that $U(x)$ is the
$S_4$-representation obtained from the quotient $S_4\to
S_4/H\cong S_3$ and the standard representation of $S_3$.  Hence,
$U(x)$ is $V_{(2,2)}$.  It is clear that $W(x,y,z)$ is the sign
representation $V_{(1,1,1,1)}$.  Lastly, one easily checks that the composition factors of
$V(x,y)$ and $V_{(2,2)} \oplus V_{(2,1,1)}$ are the same; this follows, e.g., from an explicit
computation using Brauer characters (see \cite[Chpt.~7, Def.~2.7]{modular}).  We see then that $G(R_4/K)$ has the same composition
factors as
\[
V_{(4)} \oplus V_{(3,1)}^{\oplus 3} \oplus V_{(2,2)}^{\oplus 6} \oplus V_{(2,1,1)}^{\oplus 3} \oplus V_{(1,1,1,1)};
\]
that is, if we weaken Theorem \ref{thm:main} to only require that the
two $S_n$-modules have the same composition factors, then it holds in
arbitrary characteristic for $n\leq 4$.

%\section{Examples of torsion for rings of rank $n$ over
%  $\Z$}\label{orderexample}

\section{Open questions}\label{sec:open}

%In this section we pose several questions.
There are several questions about $S_n$-closures that have
not been treated in this article, which beg for further investigation.  
First, we have the natural question:

\begin{question}
\label{q:geom}
\emph{
Is there a geometric definition of the $S_n$-closure?
}
\end{question}
The definition we have given in the introduction 
is rather algebraic.  A more geometric
definition would perhaps make various
%the functoriality of the $S_n$-closure construction 
properties of the $S_n$-closure (such as the fact that it commutes with base change!) more apparent.

% as explained at the end
%of the introduction, only applies to $n$-coverings $X/Y$.  If a more
%geometric definition were given, then presumably one would be able to
%define the $S_n$-closure of a much wider class of morphisms of
%schemes.
Second, we have only proven Theorem \ref{thm:main} in the case where
the field $K$ has characteristic prime to $n!$\,.  However, we saw in
Section~\ref{degexample} that even when $K$ has characteristic 2 or
3, the dimension of $G(R_4/K)$ remains 32, which is precisely what
Theorem~\ref{thm:main} would imply in good characteristic.  In
Section~\ref{sec:ex}, we saw in fact that 
$G(R_4/K)$ possesses 
the same composition factors in any characteristic.
%We therefore pose the following question.
Does the analogous statement hold for $G(R_n/K)$ for higher values of $n$?
\begin{question}
\label{q:char}
\emph{
Is it true, for a field $K$ of arbitrary characteristic, that 
\[ G(R_n/K) \,\,\,\,\,\,\mbox{ and }\,\,\,\,\,
\bigoplus_{\substack{\mu\triangleright\lambda\\ \mu_1=\lambda_1}}m_\lambda K_{\mu\lambda}V_\mu
\]
possess the same composition factors?
}
\end{question}
%This question stems from our observations in Sections \ref{degexample}
%and \ref{sec:ex}.  Namely, although Theorem \ref{thm:main} applies
%only when the characteristic of $K$ is prime to $n!$, we saw in
%Section \ref{degexample} that even when $K$ has characteristic 2 or 3,
%the dimension of $G(R_4/K)$ is 32, which is what Theorem
%\ref{thm:main} implies in good characteristic.  In Section
%\ref{sec:ex} we saw that the reason the dimension of $G(R_4/K)$ is
%``correct'' in bad characteristic is because $G(R_4/K)$ still has the
%``correct'' composition factors.
%We therefore pose the following question.

We have shown that the $S_n$-closure of an algebra $A$ of rank $n$
over a field $K$ has dimension $n!$ in many natural cases, and that
this dimension in any case is always bounded above by
$\dim_K(G(R_n/K))$.  What about a lower bound?  
One would guess that
the rank could never go {\it below} $n!$, although this does not seem
trivial to prove.
\begin{question}
\label{q:etale}
\emph{If $A$ is a ring of rank $n$ over a field $K$, then is the rank of
  $G(A/K)$ {at least} $n!$ ? }
\end{question}
While we do not know the answer to this question in general, we show below
that the answer is ``yes'' provided that $n$ is small and the characteristic of
$K$ is not 2 or 3:
\begin{proposition}
\label{prop:etale7}
If $n\leq7$, and $A$ is a ring of rank $n$ over a field $K$ having
characteristic not $2$ or $3$, then $G(A/K)$ has rank at least $n!$\,.
\end{proposition}
\begin{proof}
  By \cite[Cor.~6.7]{moduli} and the fact that $\mathfrak{B}_{n,K}$ is
  irreducible (\cite[Thm.~1.1]{CEVV} which assumes $K$ does not have 
  characteristic $2$ or $3$), we see that the \'etale locus is
  dense in $\mathfrak{B}_{n,K}$.  Theorem~\ref{etalecase} shows that
  if $A$ is \'etale over $K$, then the rank of $G(A/K)$ is $n!$.
  Therefore, an upper semi-continuity argument, similar to the one
  given in Theorem \ref{maxrank}, finishes the proof.
\end{proof}

%\noindent
The argument of Proposition \ref{prop:etale7} does not extend to higher values of $n$
because it is known that the \'etale locus is {\it not} dense in
$\mathfrak{B}_{n,K}$ for $n\geq 8$; see \cite[Prop.~9.6]{moduli}.

%Finally, we remark that, 
Another question stems from the following.  In the Galois theory of fields, one often
constructs Galois closures through certain natural intermediate
extensions.  Namely, suppose $L=K[x]/f(x)$ is a separable field
extension of degree $n$ with associated Galois group $S_n$, and
$\tilde L$ is the splitting field of $f$ (and thus the Galois closure
of $L$ over $K$).  Then $f$ has a root $\alpha_1$ in $L$, and $f$ has
$n$ roots $\alpha_1,\ldots,\alpha_n$ in the splitting field
$\tilde L$.
We may thus construct $\tilde L$ through a tower of extensions
$$L\,=\,L^{(1)}\,\,\subset\,\, L^{(2)}\,\,\subset\,\,\, \cdots \,\,\,\subset\,\, L^{(n)}=\,\tilde L$$ where
$L^{(r)}:=L(\alpha_1,\ldots,\alpha_r)$ has
degree $n(n-1)\cdots(n-r+1)$ over $L$.  The fields $L^{(r)}$ are
well-defined up to isomorphism and independent of the ordering of the
roots $\alpha_1,\ldots,\alpha_r$ of~$f$.  
%Is there a way to define
%these intermediate extensions for a general locally free ring $A$
%of rank $n$ over $B$?

\begin{question}{\em Let $A$ be a ring of rank $n$ over $B$.  Is there
    a construction of ``intermediate $S_n$-closures''
    $$A=G^{(1)}(A/B),\,\,\,\,\, G^{(2)}(A/B),\,\,\,\, \,\ldots\,\,\,, 
\,\,\,\,\, G^{(n)}(A/B)=G(A/B),$$
    which commute with base change and 
    such that in the case of an $S_n$-extension of fields $L/K$ of
    degree $n$, we have $G^{(r)}(L/K)\cong L^{(r)}$?}
\end{question}

\noindent
A natural method to proceed would be to construct $G^{(r)}(A/B)$ as a
quotient of $A^{\otimes r}$ by an appropriate ideal $I^{(r)}(A,B)$,
where $I^{(n)}(A,B)$ coincides with $I(A,B)\subset A^{\otimes n}$.

Finally, it is natural to ask whether Galois type closures can be
obtained for groups other than~$S_n$.  If $G\subset S_n$ is a 
permutation group on $n$ elements, then there should be an analogous 
way to define a ``$G$-closure'' for rank $n$ rings with appropriate properties.  
In the case of separable field extensions $A/B$ where $\Gal(\tilde A/B)\subset G$, this should then yield
a $B$-algebra isomorphic to $\tilde A^{|G|/\deg(A/B)}$ as in Theorem~\ref{fieldcase}.
%More generally, we ask:

\begin{question}{\em If $G\subset S_n$ is a permutation group, what is the natural class of rings/schemes for 
which functorial $G$-closures can be defined?}
\end{question}

\subsection*{Acknowledgments}

We thank O.\ Biesel, J.\ Blasiak, T.\ Church, B.\ de Smit, J.\ Ellenberg, D.\ Erman, W.\ Ho,
K.\ Kedlaya, H.\ Lenstra, J.\ Parson, B.\ Poonen, S. Sun, J-P.\ Serre, C.\ Skinner, N.\
Snyder, L.\ Taelman, B.\ Viray, and M.\ Wood for numerous valuable conversations,
comments, and suggestions which helped shape this article.  We also thank the anonymous referee for a very careful reading and extremely helpful comments.  

This paper was presented at a special Intercity Number Theory Day held in Leiden in honor of the 60th birthday of Hendrik Lenstra, and it is a pleasure to dedicate this paper to him.

%\vspace{1.65in}
\vspace{1in}

%\vspace{.65in}

%Department of Mathematics, Princeton University, Princeton, N.J. 08544

\end{document}